\documentclass[10pt]{iopart}
\usepackage{iopams}
\usepackage{setstack}
\usepackage{amsthm}
\usepackage{dsfont}
 \usepackage{comment}
 \usepackage{enumitem}
\usepackage{subcaption}
\usepackage{makebox}
\usepackage{todonotes}
\usepackage{appendix}
\newtheorem{theorem}{Theorem}[section]
\newtheorem{corollary}{Corollary}
\newtheorem{lemma}[theorem]{Lemma}
\newtheorem{example}{Example}
\newtheorem{condition}{Condition}
\newtheorem{conditions}{Condition}

\newcommand{\aksel}[1]{{\color{magenta}[Aksel: #1] }}

\begin{document}
\title[Consistent reconstruction of inclusions] 
      {A Bayesian approach for consistent reconstruction of inclusions}

\author{B M Afkham$^1$, K Knudsen$^1$, A K Rasmussen$^1$\footnote{Corresponding author} and T Tarvainen$^2$}
\address{$^1$ Department of Applied Mathematics and Computer Science, Technical University of Denmark, DK-2800 Kgs. Lyngby, Denmark}
\address{$^2$ Department of Technical Physics, University of Eastern Finland, Kuopio 70210, Finland}
\ead{akara@dtu.dk}
\begin{abstract}
    This paper considers a Bayesian approach for inclusion detection in nonlinear inverse problems using two known and popular push-forward prior distributions:  the star-shaped and level set prior distributions. We analyze the convergence of the corresponding posterior distributions in a small measurement noise limit. The methodology is general; it works for priors arising from any Hölder continuous transformation of Gaussian random fields and is applicable to a range of inverse problems. 
    The level set and star-shaped prior distributions are examples of push-forward priors under Hölder continuous transformations that take advantage of the structure of inclusion detection problems. 
   We show that the corresponding posterior mean converges to the ground truth in a proper probabilistic sense. Numerical tests on a two-dimensional quantitative photoacoustic tomography problem showcase the approach. The results highlight the convergence properties of the posterior distributions and the ability of the methodology to detect inclusions with sufficiently regular boundaries.
\end{abstract}

\section{Introduction}
The Bayesian approach to inverse problems has in recent decades generated considerable interest due to its ability to incorporate prior knowledge and quantify uncertainty in solutions to inverse problems, see \cite{kaipio2005,stuart2010}. A commonly recurring objective in inverse problems for imaging science is to recover inhomogeneities or inclusions, i.e. piecewise constant features, in a medium; applications range from cancer detection in medical imaging \cite{cherepenin2001,xu2006}  to defect detection in material science \cite{hallaji2014,fuch2021}. In a Bayesian framework, this can be tackled by designing a prior distribution that favors images with these features. \\

An optimization-based approach can address this by parametrizing the relevant subset of the image space and minimizing a functional over the preimage of this parametrization, see for example \cite{bora2017}. This is visualized in Figure \ref{fig:figure-spaces}, where we consider the parametrization $\Phi$ defined on a linear space $\Theta$ and giving rise to the subset $\Phi(\Theta)$ of the image space
\begin{eqnarray}\nonumber
    L^2_\Lambda(D) = \{\gamma \in L^2(D): \Lambda^{-1} \leq \gamma \leq \Lambda \textnormal{ a.e.}\},
\end{eqnarray}
where $D$ is a bounded and smooth domain in $\mathbb{R}^d$, $d=2,3$ and $\Lambda>0$ is a constant. 
Such approaches benefit computationally from the fact that the set of images with inclusions, i.e. $\Phi(\Theta)$, form a low-dimensional subset of the image space $L^2_\Lambda(D)$. 
In the Bayesian framework, a related approach makes use of a push-forward distribution as the prior distribution, i.e. the distribution of a transformed random element of $\Theta$. 
This often leads to strong \textit{a priori} assumptions, as the prior only gives mass to the range of the parametrization. More classical prior distributions including Laplace-type priors, see for example \cite{dashti2017}, and other heavy-tailed distributions often fail to take advantage of the low dimension of such images.\\


In this paper, we consider a Bayesian approach that captures this idea for two parametrizations used in detection of inclusions for nonlinear inverse problems: the star-shaped set and level set parametrizations. These parametrizations are studied rigorously in \cite{buithanh2014,dunlop2016,igelsias2016} and remain popular to Bayesian practitioners: we mention \cite{kaipio2005,afkham2023,carpio2020,borggaard2023,yin2022,yan2020} in the case of the star-shaped inclusions and \cite{dunlop2017,chada2018,huang2021a,huang2021b,afkham2023,reese2021} for the level set inclusions, see also references therein.\\

The solution to the inverse problem in the Bayesian setting is the conditional distribution of the unknown given data, referred to as the posterior distribution. The posterior distribution has proved to be well-posed in the sense of \cite{stuart2010} for such parametrizations.
This means that the posterior distribution continuously depends on the data in some metric for distributions.
This property implies, for example, that the posterior mean and variance are continuous with respect to the data, see \cite{dashti2017}. However, such results give no guarantee as to where the posterior distribution puts its mass. \\

A more recent framework provided in \cite{monard2021} using ideas from \cite{ghoshal2000}, see also \cite{nickl2022}, gives tools to analyze the convergence of the posterior distribution for nonlinear inverse problems. 
Such results, known as `posterior consistency', address whether the sequence of posterior distributions arising from improving data (in a small noise or large sample size limit) gives mass approximating 1 to balls centered in the ground true parameter $\gamma_0$ generating the data. Nonlinearity in the forward map and parametrization makes consistency results for Gaussian posterior distributions, as in \cite{agapiou2013}, inapplicable. Currently, the setting of \cite{monard2021} and similar approaches require smoothness of the parameter of interest. A crucial condition is that the parameter set that is given most of the mass by the prior, has  small `complexity' in the sense of covering numbers, see \cite[Theorem 2.1]{ghoshal2000} or \cite[Theorem 1.3.2]{nickl2022}. Using Gaussian priors, this parameter set is typically a closed Sobolev or Hölder norm ball, see \cite[Theorem 2.2.2]{nickl2022} or \cite{monard2021}. However, such priors do not give sufficient mass to discontinuous parameters to conclude consistency. In this paper, we aim to address this, at least partially, by parametrizing the set of discontinuous parameters from a linear space $\Theta$ of sufficiently smooth functions.\\

\begin{figure}[tbp!]
	\centering
	\scalebox{0.9}{
	\tikzset{every picture/.style={line width=0.75pt}} 

\begin{tikzpicture}[x=0.75pt,y=0.75pt,yscale=-1,xscale=1]

\draw    (149.67,158.67) .. controls (187.69,136.24) and (312.99,157.62) .. (348.76,172.38) ;
\draw [shift={(351.33,173.5)}, rotate = 203.31] [fill={rgb, 255:red, 0; green, 0; blue, 0 }  ][line width=0.08]  [draw opacity=0] (7.14,-3.43) -- (0,0) -- (7.14,3.43) -- cycle    ;
\draw   (93.67,129) -- (220.33,129) -- (220.33,250) -- (93.67,250) -- cycle ;
\draw [color={rgb, 255:red, 0; green, 0; blue, 0 }  ,draw opacity=1 ][line width=0.75]    (287.33,236.83) .. controls (354.33,194.83) and (262.98,171.98) .. (284.98,150.98) .. controls (309.98,125.98) and (356.98,136.98) .. (353.98,177.98) .. controls (351.31,219.48) and (354.65,236.15) .. (377.65,236.15) ;
\draw    (167,230.67) .. controls (188.07,203.96) and (264.23,193.74) .. (305.2,202.76) ;
\draw [shift={(307.67,203.33)}, rotate = 192.41] [fill={rgb, 255:red, 0; green, 0; blue, 0 }  ][line width=0.08]  [draw opacity=0] (5.36,-2.57) -- (0,0) -- (5.36,2.57) -- cycle    ;
\draw    (136.33,196) .. controls (182.07,175.75) and (265.92,192.63) .. (307.19,202.72) ;
\draw [shift={(309.67,203.33)}, rotate = 193.79] [fill={rgb, 255:red, 0; green, 0; blue, 0 }  ][line width=0.08]  [draw opacity=0] (7.14,-3.43) -- (0,0) -- (7.14,3.43) -- cycle    ;
\draw  [draw opacity=0][fill={rgb, 255:red, 0; green, 0; blue, 0 }  ,fill opacity=1 ] (134.84,196) .. controls (134.84,194.44) and (136.11,193.18) .. (137.67,193.18) .. controls (139.23,193.18) and (140.49,194.44) .. (140.49,196) .. controls (140.49,197.56) and (139.23,198.82) .. (137.67,198.82) .. controls (136.11,198.82) and (134.84,197.56) .. (134.84,196) -- cycle ;
\draw   (261.33,129) -- (388,129) -- (388,250) -- (261.33,250) -- cycle ;
\draw  [draw opacity=0][fill={rgb, 255:red, 0; green, 0; blue, 0 }  ,fill opacity=1 ] (164.18,230.67) .. controls (164.18,229.11) and (165.44,227.84) .. (167,227.84) .. controls (168.56,227.84) and (169.82,229.11) .. (169.82,230.67) .. controls (169.82,232.23) and (168.56,233.49) .. (167,233.49) .. controls (165.44,233.49) and (164.18,232.23) .. (164.18,230.67) -- cycle ;
\draw  [draw opacity=0][fill={rgb, 255:red, 0; green, 0; blue, 0 }  ,fill opacity=1 ] (146.84,158.67) .. controls (146.84,157.11) and (148.11,155.84) .. (149.67,155.84) .. controls (151.23,155.84) and (152.49,157.11) .. (152.49,158.67) .. controls (152.49,160.23) and (151.23,161.49) .. (149.67,161.49) .. controls (148.11,161.49) and (146.84,160.23) .. (146.84,158.67) -- cycle ;
\draw  [draw opacity=0][fill={rgb, 255:red, 0; green, 0; blue, 0 }  ,fill opacity=1 ] (309.67,204.33) .. controls (309.67,202.77) and (310.93,201.51) .. (312.49,201.51) .. controls (314.05,201.51) and (315.31,202.77) .. (315.31,204.33) .. controls (315.31,205.89) and (314.05,207.16) .. (312.49,207.16) .. controls (310.93,207.16) and (309.67,205.89) .. (309.67,204.33) -- cycle ;
\draw  [draw opacity=0][fill={rgb, 255:red, 0; green, 0; blue, 0 }  ,fill opacity=1 ] (351.16,174.16) .. controls (351.16,172.6) and (352.42,171.33) .. (353.98,171.33) .. controls (355.54,171.33) and (356.8,172.6) .. (356.8,174.16) .. controls (356.8,175.71) and (355.54,176.98) .. (353.98,176.98) .. controls (352.42,176.98) and (351.16,175.71) .. (351.16,174.16) -- cycle ;
\draw   (428.33,129) -- (555,129) -- (555,250) -- (428.33,250) -- cycle ;
\draw    (312.49,204.33) .. controls (341.91,181.31) and (451.17,191.25) .. (495.06,207.97) ;
\draw [shift={(497.67,209)}, rotate = 200.86] [fill={rgb, 255:red, 0; green, 0; blue, 0 }  ][line width=0.08]  [draw opacity=0] (7.14,-3.43) -- (0,0) -- (7.14,3.43) -- cycle    ;
\draw  [draw opacity=0][fill={rgb, 255:red, 0; green, 0; blue, 0 }  ,fill opacity=1 ] (497.67,210) .. controls (497.67,208.44) and (498.93,207.18) .. (500.49,207.18) .. controls (502.05,207.18) and (503.31,208.44) .. (503.31,210) .. controls (503.31,211.56) and (502.05,212.82) .. (500.49,212.82) .. controls (498.93,212.82) and (497.67,211.56) .. (497.67,210) -- cycle ;
\draw    (353.98,174.16) .. controls (383.4,151.14) and (425.61,147.47) .. (466.81,163.65) ;
\draw [shift={(469.33,164.67)}, rotate = 201.43] [fill={rgb, 255:red, 0; green, 0; blue, 0 }  ][line width=0.08]  [draw opacity=0] (7.14,-3.43) -- (0,0) -- (7.14,3.43) -- cycle    ;
\draw  [draw opacity=0][fill={rgb, 255:red, 0; green, 0; blue, 0 }  ,fill opacity=1 ] (469.33,165.67) .. controls (469.33,164.11) and (470.6,162.84) .. (472.16,162.84) .. controls (473.71,162.84) and (474.98,164.11) .. (474.98,165.67) .. controls (474.98,167.23) and (473.71,168.49) .. (472.16,168.49) .. controls (470.6,168.49) and (469.33,167.23) .. (469.33,165.67) -- cycle ;

\draw (101,135) node [anchor=north west][inner sep=0.75pt]    {\Large $\Theta $};
\draw (360,132) node [anchor=north west][inner sep=0.75pt]    {\Large $L_{\Lambda}^{2}$};
\draw (324,227.23) node [anchor=north west][inner sep=0.75pt]    {\large $\Phi(\Theta) $};
\draw (318,205) node [anchor=north west][inner sep=0.75pt]    {\Large $\gamma $};
\draw (149,225.22) node [anchor=north west][inner sep=0.75pt]    {\Large $\theta $};
\draw (234.33,210.23) node [anchor=north west][inner sep=0.75pt]    {\Large $\Phi $};
\draw (535,135) node [anchor=north west][inner sep=0.75pt]    {\Large $\mathcal{Y}$};
\draw (403.33,207.23) node [anchor=north west][inner sep=0.75pt]    {\Large $\mathcal{G}$};

\end{tikzpicture}
	}
    \caption{A visiualization of the parametrization $\Phi:\Theta \rightarrow L^2_\Lambda(D)$ and forward map $\mathcal{G}:L^2_\Lambda(D)\rightarrow \mathcal{Y}$.}
    \label{fig:figure-spaces}
\end{figure}
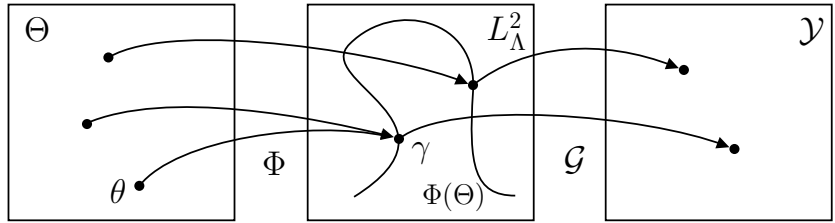
We aim to recover an element $\gamma$, which we call the \textit{image} or the \textit{physical parameter}, in a subset $\Phi(\Theta)$ of $L^2_\Lambda(D)$ for some continuous map $\Phi:\Theta\rightarrow L^2_\Lambda(D)$.
We consider a nonlinear forward map $\mathcal{G}:L^2_\Lambda(D)\rightarrow \mathcal{Y}$ mapping into a real separable Hilbert space $\mathcal{Y}$ with inner product $\langle\cdot,\cdot\rangle$ and norm $\|\cdot\|$. We refer again to Figure \ref{fig:figure-spaces} for an overview of this setup. 
This setting allows us to make use of the framework provided in \cite{monard2021}, but transfer the complexity condition from subsets of $L^2_\Lambda(D)$ to subsets of $\Theta$, see Section \ref{sec:metric-entropy}. In the context of inclusion detection, this means we can detect inclusions with sufficiently smooth boundaries. \\

Our contributions can be summarized as follows:
\begin{itemize}
    \item We present a posterior consistency result for the general setting mentioned above, when the parametrization $\Phi$ satisfies mild conditions in regularity. We use the framework provided by \cite{monard2021} extending to Hölder continuous $\mathcal{G}$ and push-forward priors. In particular, this gives an estimator, the posterior mean, which converges in probability to the true physical parameter in the small noise limit. 
    Formally, this means there is an algorithm $\hat{\gamma}$ defined for noisy measurements $Y$ depending on the noise level $\varepsilon>0$ such that
    \begin{eqnarray}\nonumber\|\hat{\gamma}(Y)-\gamma_0\|_{L^2(D)}\rightarrow 0\end{eqnarray}
    in probability as  $\varepsilon\rightarrow 0$. This statement will be made precise in Section \ref{section:2}. Furthermore, the rate of convergence is determined in part by the smoothness of elements in $\Theta$ and the regularity of the parametrization. 
    \item We show that two parametrizations for inclusion detection, a star-shaped set parametrization and a smoothened level set parametrization, satisfy the conditions for this setup. This verifies and quantifies the use of such parametrizations.
    \item We numerically verify the approach based on the two parametrizations in a small noise limit for a nonlinear PDE-based inverse problem using Markov chain Monte Carlo (MCMC) methods. We consider a two-dimensional quantitative photoacoustic (QPAT) problem of detecting an absorption coefficient. We derive a new stability estimate following \cite{choulli2021,bal2011}.
\end{itemize}

We note that the framework of \cite{monard2021} in e.g. \cite{abraham2019} and \cite{giordano2020} shows consistency for `regular link functions' $\Phi$ (defined in \cite{nickl2020b}), which are smooth and bijective. The archetypal example is $\Phi=\exp$ or a smoothened version to ensure positivity of the physical parameter $\gamma$. As we shall see, injectivity and inverse continuity are not necessary for the proofs when we want to show consistency in $L^2_\Lambda(D)$. One novelty of our work is to show that this observation has relevance: 
we seek to recover the physical parameter $\gamma$ instead of a non-physical parameter in $\Theta$ that generated it. As we shall see, a natural parametrization for star-shaped inclusions is Hölder continuous from a suitable space $\Theta$ to $L^2_\Lambda(D)$. The same holds true for a smoothened level set parametrization, which we will encounter in Section \ref{sec:level-set}.\\ 

The structure of the paper can be summarized as follows. In Section \ref{section:2}, we recall a few key elements of the Bayesian framework in a `white noise' model as outlined in \cite{stuart2010} and \cite[Section 7.4]{nickl2020}, including the notion of posterior consistency with a rate. In Section \ref{section:3}, we show that Hölder continuity of $\Phi$, some smoothness of elements in $\Phi(\Theta)$ and conditional continuity of $\mathcal{G}^{-1}$, suffice to show that the posterior mean 
converges to the ground truth $\gamma_0$ in $L^2(D)$ as the noise goes to zero. Section \ref{section:4} considers these conditions for the level set and star-shaped set parametrizations, which are well-known in the literature. In Section \ref{section:5}, we consider the two-dimensional quantitative photoacoustic tomography problem suited for piecewise constant parameter inference. Then, Section \ref{section:6} gives background to our numerical tests and results that emulate the theoretical setting of Section \ref{section:3}. We present conclusive remarks in Section \ref{sec:conclusions}.\\

In the following, we let random variables be defined on a measure space $(\Omega,\mathcal{F},\mathrm{Pr})$. For a metric space $\mathcal{Z}_1$ the Borel $\sigma$-algebra is denoted by $\mathcal{B}(\mathcal{Z}_1)$. If $F:\mathcal{Z}_1\rightarrow \mathcal{Z}_2$ is a measurable map between the measure space $(\mathcal{Z}_1,\mathcal{B}(\mathcal{Z}_1),m)$ and the measurable space $(\mathcal{Z}_2,\mathcal{B}(\mathcal{Z}_2))$, then $Fm$ denotes the push-forward measure defined by $Fm (B) = m(F^{-1}(B))$ for all $B\in \mathcal{B}(\mathcal{Z}_1)$. 
We denote by $L^2(\Omega,\mathrm{Pr})$ the space of real-valued square integrable measurable functions from $(\Omega,\mathcal{F},\mathrm{Pr})$ to $(\mathbb{R},\mathcal{B}(\mathbb{R}))$. When $\mathrm{Pr}$ is the Lesbegue measure on $\mathbb{R}$, we simply write $L^2(\Omega)$. We call probability measures defined on $\mathcal{B}(\mathcal{Z}_1)$ Borel distributions.

\section{The Bayesian approach to inverse problems}\label{section:2}
Bayesian inference in inverse problems centers around a posterior distribution. This is formulated by Bayes' rule once a prior distribution in $L^2_\Lambda(D)$ has been specified and the likelihood function has been determined by the measurement process. In this paper, we consider a `continuous' model of indirect observations
\begin{equation}\label{eq:obs}
    Y = \mathcal{G}(\gamma) + \varepsilon_n \xi,
\end{equation}
for a continuous forward map $\mathcal{G}:L^2_\Lambda(D)\rightarrow \mathcal{Y}$, where the separable Hilbert space $\mathcal{Y}$ has an orthonormal basis $\{e_k\}_{k=1}^\infty$. Here $\xi$ is `white noise' in $\mathcal{Y}$ defined below in \eref{eq:KLexp}. We denote the noise level by $\varepsilon_n:=\frac{\sigma}{\sqrt{n}}$ for some $\sigma>0$ and $n\in \mathbb{N}$, which has this convenient form to study a countable sequence of posterior distributions in decreasing noise, i.e. for growing $n$. When we write $Y$, it is understood that this depends on $n$ and $\gamma$. The rate $n^{-1/2}$ is natural: if $\mathcal{Y}$ is a subspace of Hölder continuous functions on a bounded domain, this observation model is equivalent to observing $n$ discrete point evaluations of $\mathcal{G}(\gamma)$ with added standard normal noise as $n\rightarrow \infty$, see \cite{nickl2020} and \cite[Section 1.2.3]{gine2016}.\\

Given a Borel prior distribution $\Pi$ on $L^2_\Lambda(D)$, the posterior distribution $\Pi(\cdot|Y)$ is proportional to the product of the likelihood and prior. Indeed, according to Bayes' rule, if $\mathcal{Y}$ is finite-dimensional, the posterior distribution has a density (Radon-Nikodym derivative) of the form
\begin{eqnarray}\nonumber\frac{d\Pi(\cdot|y)}{d\Pi}(\gamma) = \frac{1}{Z}  \exp\left (-\frac{1}{2\varepsilon_n^2}\|\mathcal{G}(\gamma)-y\|^2 \right), \quad \forall y \in \mathcal{Y}\end{eqnarray}
where $Z>0$ is a constant, see for example \cite{dashti2017,ghosal2017}. This is well-defined for almost all $y$ under the marginal distribution of $Y$. 
The relevance of this object emerges, when evaluating it in a realization $Y(\omega)=y$. Using inner product rules we can rewrite this as
\begin{equation}\label{eq:posteriorL2} \fl
    \Pi(B|Y) = \frac{1}{Z}\int_{B} \exp\left (\frac{1}{\varepsilon_n^2}\langle Y, \mathcal{G}(\gamma) \rangle - \frac{1}{2\varepsilon_n^2}\|\mathcal{G}(\gamma)\|^2 \right) \, \Pi(d\gamma), \quad B\in \mathcal{B}(L^2_\Lambda(D)),
\end{equation}
where the contribution of $Y$ is absorbed in the constant $Z>0$. The purpose of the following paragraphs is to argue that this formula remains valid, when $\mathcal{Y}$ is infinite-dimensional with the interpretation that $\langle Y, \mathcal{G}(\gamma) \rangle$ is a Gaussian random variable
defined by
\begin{equation}\label{eq:white-noise}
    \langle Y, y \rangle := \langle \mathcal{G}(\gamma), y \rangle + \varepsilon_n W(y),
\end{equation}
where $W$ is a white noise process on $\mathcal{Y}$ satisfying $\mathbb{E}[W(y)] = 0$ and $\mathbb{E}[W(y)W(y')]=\langle y,y' \rangle$, see \cite{gine2016}[Example 2.1.11]. To this end, let
\begin{equation}\label{eq:KLexp}
    \xi := \sum_{k=1}^\infty \xi_k e_k, \qquad \xi_k \stackrel{i.i.d.}{\sim} N(0,1),
\end{equation}
which is convergent in $\mathcal{Y}_{-}$ in the mean square sense, see \cite[Section 2.4]{dashti2017}, where $\mathcal{Y}_{-}$ is the Hilbert space $\mathcal{Y}_{-}$, see also \cite[Section 7.4]{nickl2020}, defined by
\begin{eqnarray}\nonumber\mathcal{Y}_- := \left\{f=\sum_{k=1}^\infty f_k e_k : \|f\|_{-}^2:=\sum_{k=1}^\infty \lambda_k^{2}f_k^2 < \infty \right\}\end{eqnarray}
for $\lambda_k>0$ and $\{\lambda_k\}_{k=1}^\infty \in \ell^2$. Note $\xi$ is a Gaussian random element of $\mathcal{Y}_{-}$, since it is the Karhunen-Loeve expansion of a mean zero Gaussian random element with covariance operator $K:\mathcal{Y}_{-}\rightarrow \mathcal{Y}_{-}$ defined by $Ke_k=\lambda_k^2 e_k$, see \cite{dashti2017}. 
Then $Y$ is also a $\mathcal{Y}_{-}$-valued Gaussian random element, since it is a translation of $\varepsilon_n \xi$ by an element in $\mathcal{Y}$.
We denote the distributions of $\varepsilon_n\xi$ and $Y$ in $\mathcal{Y}_{-}$  by $P_n$ and $P_n^\gamma$, respectively. We can think of $P_n^\gamma$ as the data-generating distribution indexed by $\gamma$, the physical parameter generating the data, and $n$, which controls the noise regime.\\

The likelihood function arises as the density (Radon-Nikodym derivative) of $P_n^\gamma$ with respect to $P_n$. This is a consequence of the Cameron-Martin theorem in the Hilbert space $\mathcal{Y}$.
The theorem gives the likelihood function as
\begin{eqnarray}\nonumber
    p_n^\gamma(Y) := \frac{dP_{n}^\gamma}{dP_n}(Y) = \exp \left (\frac{1}{\varepsilon_n^2}\langle Y, \mathcal{G}(\gamma)\rangle-\frac{1}{2\varepsilon_n^2}\|\mathcal{G}(\gamma)\|^2 \right),
\end{eqnarray}
here evaluated in $Y$, see \cite{gine2016}[Proposition 6.1.5]. See also a derivation in \cite[Section 7.4]{nickl2020}, for which it suffices that $\gamma \mapsto \mathcal{G}(\gamma)$ is continuous from (the standard Borel space) $L^2_\Lambda(D)$ with the $L^2(D)$-topology into $\mathcal{Y}$. \\

Then Bayes' rule \cite[p. 7]{ghosal2017} formulates \textit{a} posterior distribution as a measure in $L^2_\Lambda(D)$ as in the right-hand side of \eref{eq:posteriorL2}, well-defined for almost all $Y$. According to  \cite{ghosal2017}, this equals almost surely a Markov kernel, which we will call \textit{the} posterior distribution and also denote it by $\Pi(\cdot |Y)$. That is to say that $B\mapsto \Pi(B|Y(\omega))$ is a measure for every $\omega \in \Omega$ and $\omega \mapsto \Pi(B|Y(\omega))$ is measurable for every $B \in \mathcal{B}(L^2_\Lambda(D))$. In particular, $\omega \mapsto \Pi(B|Y(\omega))$ is a $[0,1]$-valued random variable.


\subsubsection{Convergence in probability} In preparation for the subsequent section, we recall the notion of convergence in probability. Let $t_n>0$ be a decreasing sequence going to zero. For a fixed $\gamma_0\in L^2_\Lambda(D)$ and a sequence of measurable functions $f_n: \mathcal{Y}_{-}\rightarrow \mathbb{R}$ we say that the sequence of random variables $\{f_n(Y)\}_{n=1}^\infty$ converges to $\upsilon\in \mathbb{R}$ in $P_n^{\gamma_0}$-probability with rate $t_{n}$ as $n \rightarrow \infty$ if there exists a constant $C>0$ such that
\begin{equation}\label{eq:convergenceinprob}
P_n^{\gamma_0}(y \in \mathcal{Y}_{-}: |f_n(y)-\upsilon| \leq Ct_n) \rightarrow 1,
\end{equation}
as $n\rightarrow \infty$. We consider the following two cases, where we recall that both the posterior distribution and $Y$ depend tacitly on $n$.
\begin{enumerate}
    \item     For a sequence of sets $\{B_n\}_{n=1}^\infty$ in $\mathcal{B}(L^2_\Lambda(D))$, we could claim that
    \begin{eqnarray}\nonumber\Pi(B_n|Y) \rightarrow 1 \quad \textnormal{ in $P_n^{\gamma_0}$-probability},\end{eqnarray}
    with rate $t_n$ as $n\rightarrow \infty$. That is, $f_n(Y) = \Pi(B_n|Y)$ and $\upsilon = 1$.
    If this is the case for $B_n := \{\gamma : \|\gamma-\gamma_0\|_{L^2(D)} \leq C_0r_{n}\}$ for some decreasing sequence $r_n>0$ going to zero and constant $C_0>0$, we say that the posterior distribution \textit{contracts around} or is \textit{consistent in} $\gamma_0$ at rate $r_n$.
    \item \label{example:posterior-mean} Denote by $E[\gamma|Y]$ the mean (`posterior mean') with respect to $\Pi(\cdot|Y)$. This is defined in the sense of a Bochner integral,
\begin{eqnarray}\nonumber E[\gamma|Y] := \int_{L^2_\Lambda(D)} \gamma \, \Pi(d\gamma|Y),\end{eqnarray}
which is well-defined by \cite[Theorem 2]{diestel1977}, since for all $\omega\in \Omega$ 
\begin{eqnarray}\nonumber\int_{L^2_\Lambda(D)} \|\gamma\|_{L^2_\Lambda(D)}\, d\Pi(d\gamma|Y(\omega)) \leq \Lambda\sqrt{\mathrm{vol}(D)} < \infty.\end{eqnarray}
Then $\omega \mapsto E[\gamma | Y(\omega)]$ is an $L^2_\Lambda(D)$-valued random element by the definition of the Bochner integral and by the measurability of pointwise limits of measurable functions, see \cite[Theorem 4.2.2]{dudley1989}. We could claim that
\begin{eqnarray}\nonumber
    \|E[\gamma|Y]-\gamma_0\|_{L^2(D)} \rightarrow 0 \quad \textnormal{ in $P_n^{\gamma_0}$-probability},
\end{eqnarray}
with rate $t_n$ as $n\rightarrow \infty$. That is, $f_n(Y) = \|E[\gamma|Y]-\gamma_0\|_{L^2(D)}$ and $\upsilon=0$.
\end{enumerate}

\subsection{Posterior consistency}\label{sec:post-const}
In this section we recall sufficient conditions posed in \cite{monard2021}, see also \cite{nickl2022}, such that the posterior distribution in our specific setup is consistent. More specifically, we recall for which ground truths $\gamma_0 \in L^2_\Lambda(D)$, forward models $\mathcal{G}$ and prior distributions $\Pi$
\begin{equation}\label{eq:postcontr}
    \Pi(\gamma: \|\gamma-\gamma_0\|_{L^2(D)} \leq C\tilde{r}_{n} | Y) \rightarrow 1 \quad \textnormal{ in $P_n^{\gamma_0}$-probability},
\end{equation}
as $n\rightarrow \infty$ for some positive decreasing sequence $\tilde{r}_n$ going to zero. A consequence of this result, under additional assumptions on the prior, is that the posterior mean converges to $\gamma_0$ in $P_n^{\gamma_0}$-probability, see \cite[Theorem 2.3.2]{nickl2022} or \cite[Theorem 8.8]{ghosal2017},
\begin{equation}
    \|E[\gamma|Y]-\gamma_0\|_{L^2(D)} \rightarrow 0 \quad \textnormal{ in $P_n^{\gamma_0}$-probability},
\end{equation}
with rate $r_n$ as $n\rightarrow \infty$. This is the case of \eref{example:posterior-mean} above. In the nonlinear inverse problem setting, posterior consistency in the sense of \eref{eq:postcontr} follows from a two-step procedure with the use of conditional stability estimates.

\begin{enumerate}[leftmargin=*,label={\textbf{Step \arabic*}}]
    \item \label{step:step1} The first step reduces convergence of $\{\Pi(\tilde{B}_n|Y)\}_{n=1}^\infty$ from sets of the form
    \begin{eqnarray}\nonumber\tilde{B}_n = \{\gamma \in L^2_{\Lambda}(D): \|\gamma-\gamma_0\|_{L^2(D)}\leq C\tilde{r}_n\}\end{eqnarray}
    to sets of the form
    \begin{eqnarray}\nonumber B_n = \{\gamma \in L^2_{\Lambda}(D): \|\mathcal{G}(\gamma)-\mathcal{G}(\gamma_0)\| \leq Cr_n, \gamma \in A_n\}.\end{eqnarray}
     Indeed, for specially chosen subsets $A_n \subset L^2_\Lambda(D)$ which may depend on $n$, assume we have the estimate
    \begin{equation}\label{eq:inverse-stability-est}
            \|\gamma_1-\gamma_2\|_{L^2(D)} \leq \|\mathcal{G}(\gamma_1)-\mathcal{G}(\gamma_2)\|^\nu, 
    \end{equation}
    for all $\gamma_1,\gamma_2 \in A_n$ and some $\nu >0$. Then $B_n\subset \tilde{B}_n$ and hence
    \begin{equation}\label{eq:step-2-essence}
        \Pi(B_n|Y)\leq \Pi(\tilde{B}_n|Y),
    \end{equation}
    where $\tilde{r}_n=r_n^\nu$.
    \item \label{step:step2} The second step involves showing that $\Pi(B_n|Y)$ converges to $1$ in $P_n^{\gamma_0}$-probability as $n\rightarrow \infty$. This is posterior consistency on the `forward level'. 
\end{enumerate}
Combining \ref{step:step1} and \ref{step:step2}, we find that $\Pi(\tilde{B}_n|Y)$ converges to $1$ in $P_n^{\gamma_0}$-probability as $n\rightarrow \infty$. The `conditional' stability estimate of the first step is of independent interest for many inverse problems in literature and usually requires an in-depth analysis of the inverse problem at hand. In this paper we treat first \eref{eq:inverse-stability-est} as an assumption, see Condition \ref{cond:forward-map}. Although any modulus of continuity will do for the first step in this two-step procedure, for our concrete example in photoacoustic tomography we will show a Lipschitz stability estimate that holds for all $\gamma \in L^2_\Lambda(D)$, see Section \ref{section:5}. Our main motivation for including $A_n$ in the analysis is to keep the exposition generally applicable. 

One of the contributions of \cite{monard2021,nickl2022} is to address \ref{step:step2} for a random design regression observation model using the Theorem 2.1 in \cite{ghoshal2000} and the equivalence between the distance (semi-metric)
\begin{eqnarray}\nonumber d_{\mathcal{G}}(\gamma_1,\gamma_2):= \|\mathcal{G}(\gamma_1)-\mathcal{G}(\gamma_2)\|\end{eqnarray}
and the Hellinger distance, see \cite{nickl2022}, of the data-generating distributions (corresponding to our $p_n^\gamma$). Theorem 28 in \cite{nickl2020}, see also \cite[Theorem 7.3.5]{gine2016}, adapts the proof to the observation model \eref{eq:obs}, which is what we will use. One can see this second step as showing posterior consistency in $\mathcal{G}(\gamma_0)$ at rate $r_n$ for the push-forward $\mathcal{G}\Pi(\cdot|Y)$ as in  \cite{vollmer2013}. Below, we use the covering number $N(A,d,\rho)$ for a semimetric $d$, which denotes the minimum number of closed $d$-balls of radius $\rho>0$ needed to cover $A$, see Appendix \ref{app:cov-numbers} for a precise definition. Then the condition to complete \ref{step:step2} is as follows. 
 
\begin{conditions}\label{cond:1}
Let $\Pi=\Pi_n$ be a sequence of prior distributions in $L^2_\Lambda(D)$. Let $\mathcal{G}$ be the forward model $\mathcal{G}:L^2_\Lambda(D)\rightarrow \mathcal{Y}$ and $\gamma_0\in L^2_\Lambda(D)$ the ground truth. Let $r_n$ satisfy $r_n = n^{-a}$ for some $0<a<1/2$. Suppose that,
\begin{enumerate}[label={\textbf{A.\arabic*}}]
    \item \label{cond:1.1} the prior gives enough mass to contracting balls $B_{\mathcal{G}}(\gamma_0,r_n):=\{\gamma:d_{\mathcal{G}}(\gamma,\gamma_0)\leq r_n\}$.
        \begin{equation}
            \Pi(B_{\mathcal{G}}(\gamma_0,r_n)) \geq e^{-C_1nr_n^2}, \quad C_1>0,
        \end{equation}
    \item \label{cond:1.2} there exist sets $A_n$ that are almost the support of $\Pi$ in the sense that
    \begin{equation}
        \Pi(L^2_\Lambda(D)\setminus A_n) \leq e^{-C_2nr_n^2}, \quad C_2>C_1+4,
    \end{equation}
    \item \label{cond:1.3} and that there exists a constant $m_0>0$ such that
        \begin{equation}
            \log N(A_n, d_{\mathcal{G}},m_0 r_n) \leq C_3n r_n^2, \quad C_3>0,
        \end{equation}
\end{enumerate}
all for $n$ large enough.
\end{conditions}
Condition \ref{cond:1.1} is a sufficient condition such that the denominator of the posterior distribution cannot decay too fast as $n\rightarrow \infty$. This is helpful when showing $\Pi(B_n^c |Y)\rightarrow 0$ in $P_n^{\gamma_0}$-probability as $n\rightarrow \infty$. On the other hand Condition \ref{cond:1.2} and \ref{cond:1.3} are conditions that give control over the numerator in a sense that is made precise in the proof of Theorem 2.1 in \cite{ghoshal2000} (or for example Theorem 28 in \cite{nickl2020}). It is also a trade-off; the sets $A_n$ should be large enough such that they are almost the support of the prior, but small enough such that the covering number increases sufficiently slowly when $n\rightarrow \infty$. In the general case, \ref{step:step2} is completed by the following result proved in Appendix \ref{app:testing}.
\begin{theorem}\label{thm:contr}
    Let $\Pi(\cdot|Y)$ be the sequence of posterior distributions arising for the model \eref{eq:obs} with $\gamma_0 \in L^2_\Lambda(D)$, $\mathcal{G}$ and prior distributions $\Pi=\Pi_n$ satisfying Condition \ref{cond:1} for some rate $r_n$. Then, there exists $C_0=C_0(C_2,C_3,m_0,\sigma)$ such that
    \begin{equation}\label{eq:postcontrforward}
    \Pi(B_{\mathcal{G}}(\gamma_0,C_0r_n) \cap A_n | Y) \rightarrow 1 \quad \textnormal{ in $P_n^{\gamma_0}$-probability},
    \end{equation}
    with rate $e^{-bnr_n^2}$ for all $0<b<C_2-C_1-4$ as $n \rightarrow \infty$. 
\end{theorem}

Given the preceding result, we can conclude posterior consistency in $\gamma_0$ at rate $\tilde{r}_n$ as in \ref{step:step1}, if we have a conditional stability estimate as \eref{eq:inverse-stability-est}.

\subsection{Markov chain Monte Carlo}\label{sec:mcmc}
While Section \ref{sec:post-const} concludes in an abstract way the usefulness of the posterior distribution, in this section we briefly recall methods to approximate it. We consider MCMC methods that approximate $E[\gamma|Y]$ (or other statistics) from averages of samples from a Markov chain that has the posterior distribution as its stationary distribution.
Since the composition $\mathcal{G}\circ \Phi$ maps $\Theta$ into $\mathcal{Y}$ continuously by assumption, given a prior distribution $\Pi_\theta$ in $\Theta$, there exists a posterior distribution $\Pi_\theta(\cdot|Y)$ in $\Theta$ of the form
\begin{equation}\label{eq:posterior-in-Theta} \fl
    \Pi_\theta(B|Y) := \frac{1}{Z}\int_{B} \exp\left (\frac{1}{\varepsilon_n^2}\langle Y, \mathcal{G}(\Phi(\theta)) \rangle - \frac{1}{2\varepsilon_n^2}\|\mathcal{G}(\Phi(\theta))\|^2 \right) \, \Pi_\theta(d\theta), \quad  B\in \mathcal{B}(\Theta).
\end{equation}
Naturally, if $\Pi = \Phi\Pi_\theta$, then by a change of variables  
\begin{eqnarray}\nonumber \Pi(\cdot|Y) = \Phi\Pi_\theta(\cdot|Y),\end{eqnarray}
see for example \cite[Theorem B.1]{vollmer2013}, i.e. $\theta\sim \Pi_\theta(\cdot|Y)$ implies $\Phi(\theta)\sim \Pi(\cdot|Y)$. This gives rise to the following `high-level' algorithm: given a realization $y\in \mathcal{Y}_{-}$ of $Y$, 
\begin{enumerate}[label={\arabic*.}]
    \item choose $\theta^{(0)}\in \Theta$ and $K>0$,
    \item generate $\{\theta^{(k)}\}_{k=1}^K$ in $\Theta$ using $\theta^{(0)}$ as initial condition with an MCMC method targeting $\Pi_\theta(\cdot | y)$, and
    \item return $\{\Phi(\theta^{(k))}\}_{k=1}^K$.
\end{enumerate}
For our numerical examples, we use the preconditioned Crank-Nicolson (pCN) MCMC method, see \cite{cotter2013}. This method uses only single evaluation of the log-likelihood function every iteration and is hence attractive for expensive PDE-based forward maps. It is well-defined when $\Theta$ is a Hilbert space and possesses favorable theoretical properties, see \cite{cotter2013,hairer2014}. The idea to generate samples from $\Pi(\cdot|y)$ by pushing forward samples also appears in certain reparametrizations of posterior distributions for the use of hyperparameters, see \cite{dunlop2020}. 

\section{Posterior consistency using parametrizations}
\label{section:3}
In this section, we follow \cite{monard2021,nickl2022} in their approach to satisfy Condition \ref{cond:1}. In the case where $\Pi=\Phi\Pi_\theta$ for $\Pi_\theta$ Gaussian and $ \mathcal{G}\circ \Phi$ Lipschitz continuous, the approach is the same. We give a brief recap for the case where $ \mathcal{G}\circ \Phi$ is Hölder continuous for the convenience of the reader. We tackle this by introducing three new conditions convenient for an inverse problem setting. We oppose this to Condition \ref{cond:1}, which is general and applicable in many statistical inference problems. As our base case, we assume $\Theta=H^{\beta}(\mathcal{X})$, where $\mathcal{X}$ is either the $d'$-dimensional torus or a bounded Lipschitz domain $\mathcal{X}\subset \mathbb{R}^{d'}$, $d'\geq1$ and $\beta>d'/2$. We include here the torus in our considerations, since it is a numerically convenient setting. For more general parametrizations for inclusion detection, we shall need small deviations from this setting. However, these cases will take the same starting point of $H^{\beta}(\mathcal{X})$ in Section \ref{section:4}. We begin by stating conditions on $\Phi$, $\mathcal{G}$ and $\Pi$ so that Condition \ref{cond:1} is satisfied. To do this, we introduce the following subset of $\Theta$.
\begin{eqnarray}\nonumber
    \mathcal{S}_{\beta}(M) = \{ \theta \in \Theta: \|\theta\|_{H^{\beta}(\mathcal{X})}< M \}.
\end{eqnarray}
We then require the following conditions of $\Phi$. 
\begin{condition}[On the parametrization $\Phi$]\label{cond:unifcont}
    For any $\theta_1,\theta_2\in\mathcal{S}_{\beta}(M)$ for some $M>0$, let
    \begin{equation}\label{eq:paramcont}
        \|\Phi(\theta_1)-\Phi(\theta_2)\|_{L^2(D)}\leq C_{\Phi} \|\theta_1-\theta_2\|^\zeta_{L^\infty(\mathcal{X})}
    \end{equation}
    for some constant $C_{\Phi}(M)>0$ and $0<\zeta<\infty$.
\end{condition}
That is, we require at least conditional Hölder continuity of the parametrization map $\Phi$. The $L^\infty(\mathcal{X})$ topology is not necessary for what follows and can be generalized to any $L^p$, $p\geq 1$ or $H^s$-norm, $s<\beta$. Similarly, we require conditional forward and inverse Hölder continuity of the forward map $\mathcal{G}$. 
\begin{condition}[On the forward map $\mathcal{G}$]\label{cond:forward-map}
    For any $\gamma_1,\gamma_2 \in \Phi(\mathcal{S}_{\beta}(M))$, let
    \begin{eqnarray}\nonumber
        \|\mathcal{G}(\gamma_1)-\mathcal{G}(\gamma_2)\| \leq C_{\mathcal{G}}\|\gamma_1-\gamma_2\|^\eta_{L^2(D)}
    \end{eqnarray}
    for some constants $C_{\mathcal{G}}(M)>0$ and $0<\eta<\infty$. In addition, let
    \begin{eqnarray}\nonumber
        \|\gamma_1-\gamma_2\|_{L^2(D)} \leq f(\|\mathcal{G}(\gamma_1)-\mathcal{G}(\gamma_2)\|),
    \end{eqnarray}
    for some increasing function $f:\mathbb R \to \mathbb R$, which is continuous at zero with $f(0)=0$. 
\end{condition}
We have the following condition on the prior distributions $\Pi$ we consider. They should be push-forward distributions of a scaled Gaussian prior distribution in $\Theta$. 
\begin{condition}[Prior $\Pi$]\label{cond:prior}
    Let $\Pi'_\theta$ be a centred Gaussian probability measure on $H^\beta(\mathcal{X})$, $\beta> d'/2$, with $\Pi'_\theta(H^{\beta}(\mathcal{X}))=1$. Let the reproducing kernel Hilbert space (RKHS), see \cite{ghosal2017}, $(\mathcal{H},\|\cdot\|_{\mathcal{H}})$ of $\Pi'_\theta$ be continuously embedded into $H^{\delta}(\mathcal{X})$ for some $\delta>\beta$. Then $\Pi_\theta$ is the distribution of
\begin{equation}\label{eq:scaled-prior}
    \theta  = n^{a-\frac{1}{2}} \theta', \quad \theta' \sim \Pi'_\theta
\end{equation}
for $a$ as in Condition \ref{cond:1}. Then let $\Pi=\Phi\Pi_\theta$.
\end{condition}
This gives the following structure
\begin{equation}\label{eq:structure-of-spaces}
    \mathcal{H}\subset H^\delta(\mathcal{X}) \subset H^{\beta}(\mathcal{X})=\Theta.
\end{equation}
If one chooses for example a Matérn covariance, see \cite{roininen2014}, such that $\Pi'_\theta(H^\beta(\mathcal{X}))=1$, then $\mathcal{H}=H^{\delta}(\mathcal{X})$ with $\delta=\beta+d'/2$, see Example 11.8 and Lemma 11.35 in \cite{ghosal2017} or \cite[Theorem 6.2.3]{nickl2022}. The scaling in \eref{eq:scaled-prior} essentially updates the weight of the prior term to go slower to zero. Indeed, dividing through by the factor $\varepsilon_n^{-2}$ appearing in the data-misfit term, the prior term scales as $\varepsilon_n^2 n^{1-2a} \sim r_n^2$. This term play the role of the `regularization parameter' in \cite{engl1996}. Note that $\lim_{n\rightarrow \infty} r_n= 0$ and $\lim_{n\rightarrow \infty} \varepsilon_n^2/r_n^2 = 0$, as is needed for the convergence of Tikhonov regularizers for example, see \cite[Theorem 5.2]{engl1996}. The scaling \eref{eq:scaled-prior} is also common in the consistency literature, see for example \cite{monard2021}. In our setting, it ensures that samples are with high probability in a totally bounded set $A_n$, as was called for in Condition \ref{cond:1.2} and \ref{cond:1.3}. We note for $\beta>d'/2$ that $\Pi_\theta'$ is also a Gaussian measure on the separable Banach space $C(\overline{\mathcal{X}})$ endowed with the usual supremum norm $\|f\|_{\infty}:=\sup_{x\in \overline{\mathcal{X}}}|f(x)| $. This is a consequence of a continuous Sobolev embedding and \cite[Exercise 3.39]{hairer2009}.\\

Under Condition \ref{cond:unifcont}, \ref{cond:forward-map} and \ref{cond:prior}, the lemmas in the subsequent sections ensure that Condition \ref{cond:1} is satisfied. Then we have the following theorem for posterior consistency at $\gamma_0\in \Phi(\mathcal{H})$ using the push-forward prior $\Pi = \Phi\Pi_\theta$ for $\Pi_\theta$ a Gaussian distribution satisfying Condition \ref{cond:prior}.
\begin{theorem}\label{thm:consist-in-Hbeta}
    Suppose Condition \ref{cond:unifcont}, \ref{cond:forward-map} and \ref{cond:prior} are satisfied for $\beta>d'/2$, and $\gamma_0\in \Phi(\mathcal{H})$. Let $\Pi(\cdot|Y)$ be the corresponding sequence of posterior distributions arising for the model \eref{eq:obs}. Then there exists $C_0>0$ such that
    \begin{eqnarray}\nonumber \Pi(\|\gamma-\gamma_0\|_{L^2(D)}\leq f(C_0 r_n)|Y)\rightarrow 1 \quad \textnormal{ in $P_n^{\gamma_0}$-probability},\end{eqnarray}
    where $r_n = n^{-a}$ with
    \begin{equation}\label{eq:contraction-rate-a}
        a= \frac{\eta \zeta \delta}{2\eta\zeta\delta+d'}.
    \end{equation}
    The rate of convergence in probability is $e^{-bnr_n^2}$ for any $b>0$ choosing $C_0>0$ large enough.
\end{theorem}
\begin{proof}
    Note first that Lemma \ref{lemma:small-ball} shows that Condition \ref{cond:1.1} is satisfied for some $C_1=C_1(C_\Phi,C_\mathcal{G},\zeta,\eta,d',\delta,\theta_0,\Pi_\theta')$. Given $b>0$, Lemma \ref{lemma:excess-mass-2} states that we can choose $M>C(C_2,\Pi_\theta',\delta,d')$ such that Condition \ref{cond:1.2} is satisfied and $0<b<C_2-C_1-4$. For this choice of $M$, Lemma \ref{lemma:metric-entropy-reg-sets} gives $m_0=m_0(C_\Phi,C_\mathcal{G},\zeta,\eta,M)$ and $C_3=C_3(\delta,M,d',\mathcal{X})$ such that Condition \ref{cond:1.3} is satisfied. Then, by Theorem \ref{thm:contr}, there exists $C_0(C_2,C_3,m_0)$
    \begin{eqnarray}\nonumber
    \Pi(B_{\mathcal{G}}(\gamma_0,C_0r_n) \cap A_n | Y) \rightarrow 1 \quad \textnormal{ in $P_n^{\gamma_0}$-probability},
    \end{eqnarray}
    with rate $e^{-bnr_n^2}$ as $n \rightarrow \infty$. Then the wanted result is a consequence of \eref{eq:step-2-essence}.
\end{proof}
Posterior consistency with a rate as in the preceding theorem often leads to the convergence of related estimators with the same rate, see \cite{ghosal2017}. Here, we repeat an argument found in \cite{nickl2022} to conclude that the posterior mean converges in $P_n^{\gamma_0}$-probability to $\gamma_0$ as $n\rightarrow \infty$.
\begin{corollary}\label{corollary:post-mean}
    Under the assumptions of Theorem \ref{thm:consist-in-Hbeta}, the posterior mean $E[\gamma | Y]$ in $L^2_\Lambda(D)$ satisfies for some constant $C>0$ large enough
    \begin{eqnarray}\nonumber
        \|E[\gamma|Y]-\gamma_0\|_{L^2(D)}\rightarrow 0 \quad \textnormal{ in $P_n^{\gamma_0}$-probability}
    \end{eqnarray}
    with rate $f(Cr_n)$ as $n\rightarrow \infty$.
\end{corollary}
\begin{proof}
    The proof of Theorem 2.3.2 in \cite{nickl2022} applies here, since $\Phi$ maps into $L^2_\Lambda(D)$ by assumption and hence
    \begin{eqnarray}\nonumber \fl
    \int_{L^2_\Lambda(D)} \|\gamma-\gamma_0\|_{L^2(D)}^2\, \Pi(d\gamma) = \int_{\Theta}  \|\Phi(\theta)-\Phi(\theta_0)\|_{L^2(D)}^2 \, \Pi_\theta(d\theta)\leq 4\Lambda^2 |D|.
    \end{eqnarray}
\end{proof}
 
 
\subsection{Excess mass condition \ref{cond:1.2}}
To motivate more precisely the scaling of the prior and the form of $A_n$, we recall  
 \cite[Lemma 5.17]{monard2021}: 
\begin{eqnarray}\label{eq:fernique} \nonumber
	\Pi'_\theta(\|\theta'\|_{H^\beta(\mathcal{X})}> M)\leq e^{-CM^2},
\end{eqnarray}
for all $M$ large enough and some fixed $C>0$ depending on $\Pi'_\theta$. Then 
\begin{equation}\label{eq:prior-scaling-fernique}\fl
    \Pi_\theta(\|\theta\|_{H^\beta(\mathcal{X})} > M) = \Pi'_\theta(\|\theta'\|_{H^\beta(\mathcal{X})}>  Mn^{1/2-a}) \leq e^{-CM^2n^{1-2a}} = e^{-CM^2nr_n^2}.
\end{equation}
Hence, $\Pi_\theta$ charges $\mathcal{S}_\beta(M)$ with sufficient mass in relation to Condition \ref{cond:1.2}. However, we can consider a smaller set with the same property.
Define
\begin{eqnarray}\label{eq:reg-sets} \fl
    A_n:=\Phi(\Theta_n), \qquad \Theta_n := \{\theta = \theta_1+\theta_2: \|\theta_1\|_{\infty}\leq M\bar{r}_n, \|\theta_2\|_{\mathcal{H}}\leq M\} \cap \mathcal{S}_\beta(M),
\end{eqnarray}
for $\bar{r}_n:=r_n^{\frac{1}{\eta\zeta}}$.
\begin{lemma}\label{lemma:excess-mass-2}
    If Condition \ref{cond:prior} is satisfied and $r_n=n^{-a}$ for 
    \begin{equation}\label{eq:a-ineq-ws}
        a = \frac{\eta \zeta \delta}{2\eta\zeta\delta+d'},
    \end{equation}
    then condition \ref{cond:1.2} is satisfied for $A_n$ defined by \eref{eq:reg-sets}.
\end{lemma}
\begin{proof}
    \cite[Theorem 2.2.2 and exercise 2.4.4]{nickl2022}
    shows that for $M>C(C_2,\Pi'_\theta,\delta,d')$ 
    \begin{eqnarray}\nonumber \Pi_\theta(\Theta\setminus\Theta_n)\leq e^{-C_2 nr_n^2},\end{eqnarray}
    for any given $C_2>0$, since $(\bar{r}_n n^{1/2-a})^{-b}=nr_n^2$ for $b=2d'/(2\delta-d')$. Then,
    \begin{eqnarray}\nonumber 
    \Pi(L^2_\Lambda(D)\setminus A_n) &= \Pi_\theta(\Phi^{-1}(L^2_\Lambda(D)\setminus A_n)),\\ 
    &= \Pi_\theta(\Theta \setminus \Theta_n)\leq e^{-C_2nr_n^2} \label{eq:middle-eq-prior},
\end{eqnarray}
as follows from \eref{eq:prior-scaling-fernique}.
\end{proof}
  
\subsection{Metric entropy condition \ref{cond:1.3}}\label{sec:metric-entropy}
Now we show that the sets on the form $A_n$ defined by \eref{eq:reg-sets} satisfy Condition \ref{cond:1.3}. This is straight-forward, when $\Phi$ is Hölder continuous by Lemma \ref{lemma:covering-numbers}. We also recall that an upper bound on the covering number of Sobolev norm balls is well-known, see Lemma \ref{lemma:covering-lipschitz}.
\begin{lemma}\label{lemma:metric-entropy-reg-sets}
    Suppose Condition \ref{cond:unifcont} and \ref{cond:forward-map} are satisfied. Then Condition \ref{cond:1.3} is satisfied for $A_n$ as in \eref{eq:reg-sets} and $a$ as in \eref{eq:a-ineq-ws}.
\end{lemma}
\begin{proof}
Define for $\theta'\in C(\overline{\mathcal{X}})$ and $\rho>0$ the norm ball $B_\infty(\theta',\rho):=\{\theta \in C(\overline{\mathcal{X}}): \|\theta-\theta'\|_{\infty}\leq \rho\}$ and denote by $B_\infty(\rho)$ the ball centered in $\theta'=0$. Recall \eref{eq:reg-sets}, for which we note 
$\Theta_n \subset (B_\infty(M\bar{r}_n) + \mathcal{S}_\delta(CM)) \cap \mathcal{S}_\beta(M)$ for some constant $C>0$ by Condition \ref{cond:prior}. Then applying Lemma \ref{lemma:triangle-inequality-argument} for $\rho = \overline{r}_n$
\begin{eqnarray}\nonumber N(\Theta_n, \|\cdot\|_{\infty}, 2M\overline{r}_n)\leq N(\mathcal{S}_\delta(CM), \|\cdot\|_{\infty}, M\overline{r}_n),\end{eqnarray}
Now using Lemma \ref{lemma:covering-numbers} (i) and the Hölder continuity of $\mathcal{G}\circ \Phi$ on $\mathcal{S}_\delta(M)$, there exists a constant $m_0=m_0(\eta,\zeta,C_\Phi,C_\mathcal{G},M)$ such that for any $n>0$ large enough,
\begin{eqnarray}\nonumber
\log N(A_n,d_{\mathcal{G}},m_0r_n) &\leq \log N(\Theta_n,\|\cdot\|_{\infty},2M\bar{r}_n),\\ \nonumber
 &\leq \log N(\mathcal{S}_\delta(CM),\|\cdot\|_{\infty},M\bar{r}_n),\\
 &\leq C_3\bar{r}_n^{-\frac{d'}{\delta}} = C_3nr_n^2, \label{eq:condition-on-a-covering}
 \end{eqnarray}
 where $C_3=C_3(\delta,M,d',C,\mathcal{X})$ and where we used Lemma \ref{lemma:covering-lipschitz} and \eref{eq:a-ineq-ws}.
\end{proof}

\subsection{Small ball condition  \ref{cond:1.1}}\label{sec:gaussprior}
In this section, we consider the strong assumption that $\gamma_0\in \Phi(\mathcal{H})$. We refer the reader to \cite{vaart2009} for a more general case where $\theta_0$ is only in the closure of $\mathcal{H}$ in $\Theta$. However, this extension is not immediately compatible with the scaling \eref{eq:scaled-prior}. What follows in this section is based on the work \cite{nickl2022}. We extend this to the case of Hölder continuous maps $\mathcal{G}\circ \Phi$ in a straight-forward manner. 
Below we need the scaled RKHS $\mathcal{H}_{n}:=n^{a-1/2} \mathcal{H} = \{n^{a-1/2} h: h\in \mathcal{H}\}$, see Condition \ref{cond:prior}, with norm
\begin{eqnarray}\nonumber\|h\|_{\mathcal{H}_{n}} = n^{1/2-a}\|h\|_{\mathcal{H}}.\end{eqnarray}
This is the RKHS associated with $\Pi_\theta$, see \cite{hairer2009} or \cite[Exercise 2.6.5]{gine2016}.

\begin{lemma}\label{lemma:small-ball}
Let $\Pi$ satisfy Condition \ref{cond:prior} and let $\gamma_0=\Phi(\theta_0)$ for some $\theta_0\in \mathcal{H}$. If Condition \ref{cond:unifcont} and \ref{cond:forward-map} are satisfied, then Condition \ref{cond:1.1} is satisfied for $a$ as in \eref{eq:a-ineq-ws}.
\end{lemma}
\begin{proof}
	For $R>0$ large enough depending on $\theta_0$ and $\Pi_\theta'$, we have by Condition \ref{cond:unifcont} and \ref{cond:forward-map},
		\begin{eqnarray}\fl \nonumber
		\{\theta \in \Theta: d_{\mathcal{G}}(\Phi(\theta),&\Phi(\theta_0))\leq r_{n}\}\\\nonumber
			&\supset \{\theta\in \Theta: d_{\mathcal{G}}(\Phi(\theta),\Phi(\theta_0))\leq r_{n}\} \cap \mathcal{S}_{\beta}(R),\\\nonumber
			&\supset \{\theta\in \Theta: \|\Phi(\theta)-\Phi(\theta_0)\|_{L^2(D)}\leq Cr_{n}^{1/\eta}, \|\theta\|_{H^{\beta}(\mathcal{X})}\leq R\} \\
			&\supset \{\theta\in \Theta: \|\theta-\theta_0\|_{\infty} \leq C\bar{r}_{n}, \|\theta-\theta_0\|_{H^{\beta}(\mathcal{X})}\leq \tilde R\}, \label{eq:rhs-small-ball}
		\end{eqnarray}
		where $C=C(\eta,\zeta,C_{\mathcal{G}},C_{\Phi},R)$ and $\tilde{R}=R-\|\theta_0\|_{H^\beta(\mathcal{X})}$, and where we used the triangle inequality. Note also $\tilde{\Pi}_\theta(\cdot) = \Pi_\theta(\cdot+\theta_0)$ is a Gaussian measure in the separable Hilbert space $H^\beta(\mathcal{X})$. In addition, a closed norm ball in $H^\beta(\mathcal{X})$ is a closed subset of $H^\beta(\mathcal{X})$ and so is $\{\theta\in H^\beta(\mathcal{X}): \|\theta\|_{\infty}\leq C\overline{r}_n\}$ by a Sobolev embedding.
  Then we can apply the Gaussian correlation inequality \cite[Theorem 6.2.2]{nickl2022} to \eref{eq:rhs-small-ball} so that
  \begin{eqnarray}\fl \nonumber
      \Pi_\theta(d_{\mathcal{G}}(\Phi(\theta),\Phi(\theta_0))\leq r_n) &\geq
      \Pi_\theta(\|\theta-\theta_0\|_\infty \leq C\overline{r}_n, \|\theta-\theta_0\|_{H^\beta(\mathcal{X})}\leq \tilde{R}),\\\nonumber
      &= \tilde{\Pi}_\theta(\|\theta\|_\infty \leq C\overline{r}_n, \|\theta\|_{H^\beta(\mathcal{X})}\leq \tilde{R}),\\
      &\geq \tilde{\Pi}_\theta(\|\theta\|_{\infty}\leq C\bar{r}_{n})\tilde{\Pi}_\theta(\|\theta\|_{H^\beta(\mathcal{X})}\leq \tilde{R}).\label{eq:gaussian-correlation}
  \end{eqnarray}
  To each of the factors in the right-hand side of \eref{eq:gaussian-correlation} we apply \cite[Corrollary 2.6.18]{gine2016} to the effect that for large $n$
		\begin{eqnarray}\nonumber
			\Pi_\theta(d_{\mathcal{G}}(\Phi(\theta),&\Phi(\theta_0))\leq r_{n})\\ \nonumber
			&\geq e^{-\|\theta_0\|_{\mathcal{H}_{n}}^2}\Pi_\theta(\|\theta\|_{\infty} \leq C\bar{r}_{n})\Pi_\theta(\|\theta\|_{H^\beta(\mathcal{X})}\leq \tilde R),\\
			&\geq e^{-C' n^{1-2a}}\Pi'_\theta(\|\theta'\|_{\infty}\leq C\bar{r}_nn^{1/2-a}), \nonumber
		\end{eqnarray}
		for $C'=C'(\theta_0,\Pi'_\theta)$ using also that $\Pi_\theta(\|\theta\|_{H^{\beta}(\mathcal{X})}\leq \tilde R)\leq 1/2 $ for $R$ large enough as follows from \eref{eq:prior-scaling-fernique}. The rest of the argument follows \cite[Lemma 11]{abraham2019} and uses \cite[Theorem 1.2]{li1999}, see also Lemma \ref{lemma:covering-lipschitz}, and the continuous embedding $\mathcal{H}\subset H^\delta(\mathcal{X})$ to conclude
        \begin{eqnarray}\nonumber
	       \Pi'_\theta(\|\theta'\|_{\infty} \leq C\bar{r}_nn^{1/2-a}) &\geq e^{-C''(\bar{r}_nn^{1/2-a})^{-b}},\\\nonumber
			&= e^{-C''nr_n^2}
 		\end{eqnarray} 
  with $C''=C''(C,C')$ and $b=\frac{2d'}{2\delta-d'}$, which fits the choice \eref{eq:a-ineq-ws} of $a$.
  \end{proof}

\section{Parametrizations for inclusions} 
\label{section:4} 
In this section, we make use of Theorem \ref{thm:consist-in-Hbeta} for two specific parametrizations suited for inclusion detection: a star-shaped set parametrization and a level set parametrization. These are parametrizations on the form
\begin{equation}\label{eq:paramform}
    \Phi(\theta) = \sum_{i=1}^{\mathcal{N}} \kappa_i \mathds{1}_{A_i(\theta)}
\end{equation}
for some Lebesgue measurable subsets $A_i(\theta)$ of $\mathbb{R}^d$ and constants $\kappa_i >0$ for $i=1,\ldots,\mathcal{N}$, which we denote collectively as $\mathbf{\kappa} = \{\kappa_i\}_{i=1}^\mathcal{N}$. Since we consider parametrizations that map into $L^2_\Lambda(D)$, we will implicitly consider $\Phi(\theta)$ as the restriction of the right-hand side of \eref{eq:paramform} to $D$. Note that recovering parameters on this form requires that we know \textit{a priori} the parameter values $\kappa_i$. However, this could further be modelled into the prior. In the following, we construct $A_i(\theta)$ as star-shaped sets and level sets. 
\subsection{Star-shaped set parametrization}
We start by considering the parametrization for a single inclusion, i.e. $\mathcal{N}=1$. For simplicity of exposition, we consider the star-shaped sets in the plane, although it is straight-forward to generalize to higher dimensions. Let $\varphi$ be a continuously differentiable $2\pi$-periodic function. We can think of $\theta: \mathbb{T}\rightarrow \mathbb{R}$ as a function defined on the 1-dimensional torus $\mathbb{T}:=\mathbb{R}/2\pi \mathbb{Z}$. The boundary of the star-shaped set is a deformed unit circle: for a point $x$ in $D$ it takes for  $v(\vartheta):=(\cos\vartheta,\sin\vartheta)$ the form
\begin{eqnarray}\nonumber
    \partial A(\theta)=x + \{\exp(\theta(\vartheta)) v(\vartheta),\,\, 0\leq \vartheta \leq 2\pi\},
\end{eqnarray}
Then we write
\begin{equation}\label{eq:star-shaped-inclusion}
    A(\theta)=x+\{s\exp(\theta(\vartheta)) v(\vartheta), 0\leq s \leq 1, 0\leq \vartheta \leq 2\pi\}.
\end{equation}
Let $\kappa_1,\kappa_2>0$ and define
\begin{equation}\label{eq:single-inclusion}
    \Phi(\theta):=\kappa_1 \mathds{1}_{A(\theta)}+\kappa_2.
\end{equation}
We have the following conditional continuity result, where we for simplicity fix $x\in D$. 
\begin{lemma}\label{lemma:star-shape-cont}
    Let $\theta_1,\theta_2 \in H^\beta(\mathbb{T})$ and $\|\theta_i\|_{H^\beta(\mathbb{T})}\leq M$ with $\beta>3/2$ for $i=1,2$. Then
    \begin{eqnarray}\nonumber \|\Phi(\theta_1)-\Phi(\theta_2)\|_{L^2(D)}\leq C\|\theta_1-\theta_2\|_{L^\infty(\mathbb{T})}^{1/2},\end{eqnarray}
    where $C$ only depends on $M$ and $\kappa_1$.
\end{lemma}
\begin{proof}
By the translation invariance of the Lebesgue measure, it is sufficient to bound the area of the symmetric difference $A(\theta_1)\Delta A(\theta_2):= (A(\theta_1) \setminus  A(\theta_2)) \cup (A(\theta_2)\setminus A(\theta_1)) $ for $x=0$. We parameterize this planar set using $K:[0,1]\times [0,2\pi] \rightarrow \mathbb{R}^2$, defined by
    \begin{eqnarray}\nonumber K(s,\vartheta) = [s\exp(\theta_1(\vartheta))+(1-s)\exp(\theta_2(\vartheta))]v(\vartheta).\end{eqnarray}

 Note that $\|\theta_i\|_{H^\beta(\mathbb{T})}\leq M$ implies $\|\theta_i\|_{C^1(\mathbb{T})}\leq CM$ by a continuous Sobolev embedding. We have
    \begin{eqnarray}\nonumber 
        \frac{\partial K}{\partial s}(s,\vartheta) &= [\exp(\theta_1(\vartheta))-\exp(\theta_2(\vartheta))]v(\vartheta),\\ \nonumber
        \left|\frac{\partial K}{\partial \vartheta}(s,\vartheta)\right| &\leq C(M),
    \end{eqnarray}
    and the well-known change of variables formula,
    \begin{eqnarray}\fl \nonumber
        \mathrm{vol}(A(\theta_1)\Delta A(\theta_2)) &= \int_{0}^1 \int_{0}^{2\pi} |JK(s,\vartheta)| \, d\vartheta \, ds,\\\nonumber
        &\leq C(|(\partial_s K(s,\vartheta))_1||(\partial_\vartheta K(s,\vartheta))_2|+ |(\partial_s K(s,\vartheta))_2||(\partial_\vartheta K(s,\vartheta))_1|),\\\nonumber
        &\leq C(M)|e^{\theta_1(\vartheta)}-e^{\theta_2(\vartheta)}|,\\\nonumber
        &\leq C(M) \|\theta_1-\theta_2\|_{L^\infty(\mathbb{T})},
    \end{eqnarray}
    where $|JK(s,\vartheta)|$ is the determinant of the Jacobian of the map $K$. In the last line, we used that $z\mapsto \exp(z)$ is locally Lipschitz as follows from the mean value theorem.
\end{proof}
Using the triangle inequality for the symmetric difference and the main result of \cite{schymura2014}, we would also have an estimate on the continuity of $\Phi$ as defined on $D\times H^\beta(\mathbb{T})$, i.e. on elements $(x,\theta)$. We could then endow $D\times H^\beta(\mathbb{T})$ with a product prior which straight-forwardly satisfies Condition \ref{cond:1.1}. For simplicity we skip this extension. Instead, we gather the following conclusion that follows directly from Theorem \ref{thm:consist-in-Hbeta} and Corollary \ref{corollary:post-mean}.

\begin{theorem}\label{thm:star1}
    Suppose Condition \ref{cond:forward-map} is satisfied for $\beta>3/2$. Let $\gamma_0=\Phi(\theta_0)$ for $\theta_0\in \mathcal{H}$. Let $\Pi(\cdot|Y)$ be the corresponding sequence of posterior distributions arising for the model \eref{eq:obs} and prior $\Pi = \Phi\Pi_\theta$ satisfying Condition \ref{cond:prior}. Then there exists $C>0$ such that
    \begin{equation}
        \|E[\gamma|Y]-\gamma_0\|_{L^2(D)}\rightarrow 0 \quad \textnormal{ in $P_n^{\gamma_0}$-probability}
    \end{equation}
    with rate $f(Cn^{-a})$ as $n\rightarrow \infty$, where
    \begin{eqnarray}\nonumber a=\frac{\eta \delta}{2\eta\delta+2}.\end{eqnarray}    
\end{theorem}
Note that this is the rate of \eref{eq:contraction-rate-a} with $\zeta = 1/2$ and $d'=1$. Clearly this convergence rate takes into account that a smooth star-shaped inclusion belongs to a low-dimensional subset of $L^2_\Lambda(D)$. One can think of this fast convergence rate (compared to Gaussian priors directly in $L^2(D)$) as an expression of uncertainty reduction. 
Parameters $\gamma\in L^2_\Lambda(D)$ on the form \eref{eq:single-inclusion} carry some regularity. Indeed, using results in \cite{sickel1999,faraco2013} showing $\alpha$-Sobolev regularity for $0<\alpha<1/2$ reduces to giving an upper bound of the area of the $\varepsilon$-tubular neighborhood of $\partial A(\theta)$ with respect to $\varepsilon$. This is provided by Steiner's inequality, see \cite{makai1959}, for $d=2$, or more generally by Weyl’s tube formula, see \cite{gray2004}, when $d\geq 2$. Then $\|\Phi(\theta)\|_{H^{\alpha}(D)}\leq C(M,D,\alpha)$ for $\|\theta\|_{H^\beta(\mathbb{T})}\leq M$.

\subsubsection{Multiple inclusions}
The case of multiple star-shaped inclusions is a straight-forward generalization using the triangle inequality. We consider for $\mathcal{N}\geq 1$, the map 
\begin{eqnarray}\nonumber \Phi:(H^\beta(\mathbb{T}))^{\mathcal{N}} \rightarrow L^2_\Lambda(D)\end{eqnarray}
as in \eref{eq:paramform} with $A_i(\theta)=A(\theta_i) + x_i$ from $A$ in \eref{eq:star-shaped-inclusion} with $x=0$, $x_i\in D$, and where we set $\theta=(\theta_1,\ldots,\theta_{\mathcal{N}})$. We denote $\|\cdot\|_{\mathcal{N}}$ the direct product norm associated with the norm on $L^\infty(\mathbb{T})$, i.e. 
\begin{eqnarray}\nonumber \|\theta\|_{\mathcal{N}} = \max\left(\|\theta_1\|_{L^\infty(\mathbb{T})},\ldots, \|\theta_\mathcal{N}\|_{L^\infty(\mathbb{T})}\right).\end{eqnarray}
We have the following continuity result.
\begin{lemma}\label{lemma:Ninclusions}
    Let $\theta_i, \tilde{\theta}_i \in H^\beta(\mathbb{T})$ with $\|\theta_i\|_{H^\beta(\mathbb{T})}\leq M$, $\|\tilde{\theta}_i\|_{H^\beta(\mathbb{T})}\leq M$ for $i=1,\ldots,\mathcal{N}$. For $\theta=(\theta_1,\ldots,\theta_{\mathcal{N}})$ and $\tilde{\theta}=(\tilde{\theta}_1,\ldots,\tilde{\theta}_\mathcal{N})$ we have
    \begin{eqnarray}\nonumber \|\Phi(\theta)-\Phi(\tilde{\theta})\|_{L^2(D)}\leq C\|\theta-\tilde{\theta}\|_{\mathcal{N}}^{1/2},\end{eqnarray}
    where $C$ only depends on $M$, $\kappa$ and $\mathcal{N}$.
\end{lemma}
\begin{proof}
    Using the triangle inequality and Lemma \ref{lemma:star-shape-cont},
    \begin{eqnarray}\nonumber
        \|\Phi(\theta)-\Phi(\tilde{\theta})\|_{L^2(D)}^2 &= \left\|\sum_{i=1}^{\mathcal{N}} \kappa_i(\mathds{1}_{A_i(\theta)}-\mathds{1}_{A_i(\tilde{\theta})})\right\|_{L^2(D)}^2 ,\\\nonumber
        &\leq C \left(\sum_{i=1}^{\mathcal{N}} \|\mathds{1}_{A_i(\theta)}-\mathds{1}_{A_i(\tilde{\theta})}\|_{L^2(D)}\right)^2,\\\nonumber
        &\leq C\left(\sum_{i=1}^{\mathcal{N}} \|\theta_i-\tilde{\theta}_i\|_{L^\infty(\mathbb{T})}^{1/2}\right)^2,\\\nonumber
        &\leq C\|\theta-\tilde{\theta}\|_\mathcal{N}, 
    \end{eqnarray}
    by the equivalence of the $p$-norms $p>0$ on $\mathbb{R}^{\mathcal{N}}$.
\end{proof}
Parallel to the remark before Lemma \ref{lemma:star-shape-cont}, we mention that a statement similar to Lemma \ref{lemma:Ninclusions} holds true for a map $\Phi$ defined on $(D\times H^{\beta}(\mathbb{T}))^{\mathcal{N}}$, if we in addition wish to infer $x_1,\ldots, x_\mathcal{N}$. In preparation for the main result of this section let us change notation to suit the current setting. Let
\begin{equation}\label{eq:Theta-S-def} \fl
    \Theta = H^\beta(\mathbb{T})^{\mathcal{N}} \quad \textnormal{ and } \quad \mathcal{S}_\beta(M)=\{\theta\in\Theta:\|\theta_i\|_{H^\beta(\mathbb{T})}<M, i=1,\ldots,\mathcal{N}\}.
\end{equation}
We then endow $\Theta$ with a (product) prior distribution of $\Pi_\theta$ satisfying Condition \ref{cond:prior}:
\begin{equation}\label{eq:product-prior}
    \tilde{\Pi}_\theta = \otimes_{i=1}^\mathcal{N} \Pi_\theta \quad \textnormal{satisfying} \quad \tilde{\Pi}_\theta(B) = \Pi_\theta(B_1)\ldots \Pi_\theta(B_{\mathcal{N}}),
\end{equation}
for $B = B_1\times \ldots \times B_{\mathcal{N}} \in \mathcal{B}(H^\beta(\mathbb{T}))^{\mathcal{N}} = \mathcal{B}(H^\beta(\mathbb{T})^{\mathcal{N}})$. The last equality is found in for example \cite[Lemma 1.2]{kallenberg2021}. For this prior, we have the following result, which is accounted for in Appendix \ref{sec:extra-proofs}.
\begin{theorem}\label{thm:star}
    Suppose Condition \ref{cond:forward-map} is satisfied for $\mathcal{S}_\beta(M)$ as in \eref{eq:Theta-S-def} for $\beta>3/2$. Let $\gamma_0=\Phi(\theta_0)=\Phi(\theta_{0,1},\ldots,\theta_{0,\mathcal{N}})$ for $\theta_{0,i} \in \mathcal{H}$, $i=1,\ldots,\mathcal{N}$. Let $\Pi(\cdot|Y)$ be the corresponding sequence of posterior distributions arising for the model \eref{eq:obs} and prior $\Pi = \Phi\tilde{\Pi}_\theta$ for \eref{eq:product-prior}. Then there exists $C>0$ such that
    \begin{equation}
        \|E[\gamma|Y]-\gamma_0\|_{L^2(D)}\rightarrow 0 \quad \textnormal{ in $P_n^{\gamma_0}$-probability}
    \end{equation}
    with rate $f(Cn^{-a})$ as $n\rightarrow \infty$, where
    \begin{eqnarray}\nonumber a=\frac{\eta \delta}{2\eta\delta+2}.\end{eqnarray}    
\end{theorem}
Note that this is the rate as of Theorem \ref{thm:star1}, i.e. the rate does not depend on the number of inclusions; this dependence appears in the constant $C$.

\subsection{Level set parametrization}\label{sec:level-set}
In this section, we consider the level set parametrization of piecewise constant functions. The simplest case is to compose a given continuous function $\theta:\mathcal{X}\rightarrow \mathbb{R}$, for $\mathcal{X}\supset D$, i.e. $d=d'=2,3$, with the Heaviside function $H(z)=\mathds{1}_{z\geq 0}(z)$ as
\begin{eqnarray}\nonumber \gamma(x)=\Phi(\theta)(x)=\kappa_1 H(\theta(x))+\kappa_2,\end{eqnarray}
for $\kappa_1,\kappa_2>0$. However, $\Phi:H^\beta(\mathcal{X})\rightarrow L^2_\Lambda(D)$ is not uniformly Hölder continuous on $\mathcal{S}_\beta(M)$ for any $\beta,M>0$ and hence does not satisfy Condition \ref{cond:unifcont}. Indeed, if $|\nabla \theta|$ is small near the set $\{x:\theta(x)=0\}$, small changes in $\theta$ can lead to big changes in $\gamma$. A lower bound on $|\nabla\theta|$ near this set suffices, as can be seen from the implicit function theorem, see Lemma \ref{lemma:level-set-app}. This type of condition also appears in level set estimation of probability densities, see \cite{walther1997}. We illustrate this phenomenon by the following two-dimensional example.

\begin{example}\label{example:nonuniform}
    Let $\mathcal{X}=D=B(0,1/2)$ the two-dimensional disc of radius $1/2$. Take as $\theta_{(n)}$ the radially symmetric functions $\theta_{(n)}(r,\vartheta) = \frac{1}{n}+r^{2n}$ and $\tilde{\theta}_{(n)}=-\theta_{(n)}$ for $0\leq r\leq 1$ and $0\leq \vartheta\leq 2\pi$. It is clear that $\theta_{(n)},\tilde{\theta}_{(n)} \in \mathcal{S}_1(M)$ for all $n \in \mathbb{N}$, and that
    \begin{eqnarray}\nonumber
       \|\theta_{(n)}-\tilde{\theta}_{(n)}\|_{L^\infty(\mathcal{X})} &\leq 2\|n^{-1}\|_{L^\infty(\mathcal{X})} + 2\|r^{2n}\|_{L^\infty((0,1/2))},\\\nonumber
       &\leq 2n^{-1} + 2^{1-2n} \rightarrow 0
    \end{eqnarray}
    as $n\rightarrow \infty$. However $\Phi(\theta_{(n)})=\kappa_1$ and $\Phi(\tilde{\theta}_{(n)})=\kappa_2$ so $\|\Phi(\theta_{(n)})-\Phi(\tilde{\theta}_{(n)})\|_{L^2(D)}=|\kappa_2-\kappa_1|.$ 
\end{example}
The example is easy to extend to the more general case where the $L^\infty$-norm is replaced with the $C^k$-norm. Note also that for fixed $\theta_{(n)}=\theta$, we have continuity of $\Phi$ in this particular example. This fact generalizes to continuity of $\Phi$ in functions $\theta$ that do not have critical points on $\{x:\theta(x)=0\}$.
However, for the stronger Condition \ref{cond:unifcont}, it is not obvious how much mass Gaussian distributions give to functions whose gradient is lower bounded away from zero near $\{x: \theta(x)=0\}$. For this reason, we take a different approach. We define an approximation $\Phi_\epsilon$ of $\Phi$ for which Condition \ref{cond:unifcont} is satisfied. This gives an approximate posterior distribution that contracts around $\gamma_0^\epsilon = \Phi_\epsilon(\theta_0)$. We shall see that if we take $\epsilon=n^{-k}$ for some $k\in (0,1)$, then the approximation properties of $\Phi_\epsilon$ to $\Phi$ and a triangle inequality argument ensure we have consistency at $\gamma_0=\Phi(\theta_0)$.
To this end, consider the continuous approximation $H_\epsilon$ of the Heaviside function 
\begin{equation}
H_\epsilon(z):=  \cases{ 
0 &if $z<-\epsilon$,\\
\frac{1}{2\epsilon}z+\frac{1}{2} &if $-\epsilon\leq z < \epsilon$,\\
1 &if $\epsilon \leq z$.
}
\end{equation}
We want to note two straight-forward properties of $H_\epsilon$:
\begin{equation}\label{eq:Hdelta-cont}
    |H_\epsilon(z)-H_\epsilon(\tilde{z})| \leq \frac{1}{2\epsilon}|z-\tilde{z}|, \quad \textnormal{ for all } z,\tilde{z} \in \mathbb{R},
\end{equation}
and
\begin{equation}\label{eq:Hdelta-approx}
    |H_\epsilon(z)-H(z)|\leq \frac{1}{2}\mathds{1}_{(-\epsilon,\epsilon)}(z), \quad \textnormal{ for all } z \in \mathbb{R}.
\end{equation}
We could even consider a smooth approximation for $H_\epsilon$, as in \cite{reese2021}, but this is not necessary for our case. To construct the continuous level set parametrization, take constants $\mathbf{c}=\{c_i\}_{i=1}^{\mathcal{N}}$ satisfying 
\begin{eqnarray}\nonumber -\infty = c_0 < c_1 <\ldots < c_\mathcal{N} = \infty\end{eqnarray}
for some $\mathcal{N}\in \mathbb{N}$. Given a continuous function $\theta: D\rightarrow \mathbb{R}$ define
\begin{eqnarray}\nonumber A_i(\theta) :=\{x\in D: c_{i-1} \leq \theta(x) < c_i \}, \quad i=1,\ldots,\mathcal{N},\end{eqnarray}
and let $\Phi$ be of the form \eref{eq:paramform}. The corresponding approximate level set parametrization is then
\begin{equation}\label{eq:smooth-levelset}
    \Phi_\epsilon(\theta):=\sum_{i=1}^{\mathcal{N}} \kappa_i [ H_\epsilon(\theta-c_{i-1})- H_\epsilon(\theta-c_{i})],
\end{equation}
where we define $H_\epsilon(z-c_0)=1$ and $H_\epsilon(z-c_{\mathcal{N}})=0$ for any $z\in \mathbb{R}$. One can check that $\Phi_\epsilon$ coincides with $\Phi$, when $\epsilon = 0$. 
Motivated by Example \ref{example:nonuniform} and the property that stationary Gaussian random fields have almost surely no critical points on their level sets, we define the admissible level set functions as 
\begin{eqnarray}\nonumber H^{\beta}_{\diamond}(\mathcal{X}) := H^\beta(\mathcal{X}) \cap \bigcap_{i=1}^{\mathcal{N}-1} T_{c_i}, \quad \beta>2+\frac{d'}{2},\end{eqnarray}
where
\begin{eqnarray}\nonumber T_c := \{\theta \in C^2(\overline{\mathcal{X}}): \exists x \in \mathcal{X}, \theta(x) = c, |(\nabla \theta)(x)| = 0\}^{\complement}.\end{eqnarray}
Indeed, according to \cite[Proposition 6.12]{azais2009}, for each fixed $c\in \mathbb{R}$ we have
\begin{equation}\label{eq:level-set-critical-points}
    \Pi_\theta'(T_c) = 1 \quad \textnormal{ and hence } \quad \Pi_\theta'(H_\diamond^\beta(\mathcal{X})) = 1
\end{equation}
if $\Pi_\theta'(C^2(\overline{\mathcal{X}}))=1$ and the covariance function associated with $(\theta(x):x\in \mathcal{X})$ for $\theta\sim \Pi_\theta'$ is stationary. This is permitted since $((\theta(x),\partial_1\theta(x),\ldots,\partial_{d'}\theta(x)):x\in \mathcal{X})$ is a Gaussian process, see for example \cite[Section 9.4]{rasmussen2006}. Note also that it is known that $T_c\in \mathcal{B}(C^2(\overline{D}))$ since $\{\theta\in C^2(\overline{D}): |\theta(x)-c|+|(\nabla \theta)(x)|\geq 1/n, \forall x\in D\}$ is a Borel set.

\begin{lemma}\label{lemma:4.5} We have the following:
\begin{enumerate}[label=(\roman*)]
  \item If $\theta_0 \in H^\beta_\diamond(\mathcal{X})$, then for $\beta>1+d'/2$ and $\epsilon>0$ sufficiently small
  \begin{eqnarray}\nonumber \|\Phi_\epsilon(\theta_0)-\Phi(\theta_0)\|_{L^2(D)}\leq C(\theta_0,\mathcal{X},D,\mathbf{c})\epsilon^{1/2}.\end{eqnarray}
  \item For any $\theta,\tilde{\theta} \in H^2(\mathcal{X})$,
  \begin{eqnarray}\nonumber \|\Phi_\epsilon(\theta)-\Phi_\epsilon(\tilde{\theta})\|_{L^2(D)} \leq C(\mathbf{\kappa},\mathcal{N},D) \epsilon^{-1}\|\theta-\tilde{\theta}\|_{L^\infty(D)}.\end{eqnarray}
\end{enumerate}
\end{lemma}
\begin{proof}
    $(i)$ Note first
    \begin{eqnarray}\nonumber \fl \Phi_\epsilon(\theta_0)-\Phi(\theta_0) = \sum_{i=1}^\mathcal{N} \kappa_i [(H_\epsilon(\theta_0-c_{i-1})-H(\theta_0-c_{i-1})) -\\ \nonumber
    (H_\epsilon(\theta_0-c_i) -H(\theta_0-c_{i})) ].\end{eqnarray}
    By the triangle inequality and \eref{eq:Hdelta-approx}
    \begin{eqnarray}\nonumber \fl \|\Phi_\epsilon(\theta_0)-\Phi(\theta_0)\|_{L^2(D)} \leq \sum_{i=1}^{\mathcal{N}} \kappa_i(\|\mathds{1}_{(-\epsilon,\epsilon)}(\theta_0-c_{i-1})\|_{L^2(D)} + \|\mathds{1}_{(-\epsilon,\epsilon)}(\theta_0-c_{i})\|_{L^2(D)}) \end{eqnarray}
    It is clear that $\mathds{1}_{(-\epsilon,\epsilon)}(\theta_0(x)-c_{i-1}) = \mathds{1}_{V_\epsilon}(x)$ with
    \begin{equation}\label{eq:Veps}
        V_\epsilon := \{x\in \mathcal{X}: |\theta_0(x)-c_{i-1}| < \epsilon \}.
    \end{equation}
    By Lemma \ref{lemma:level-set-app} $|V_\epsilon|\leq C(\theta_0,c_{i-1},\mathcal{X})\epsilon$, and hence the wanted result follows by repeated application.
    
    \noindent $(ii)$ Again by the triangle inequality and now \eref{eq:Hdelta-cont} we have
    \begin{eqnarray}\nonumber
        \|\Phi_\epsilon(\theta)-\Phi_\epsilon(\tilde{\theta})\|_{L^2(D)} &= \sum_{i=1}^\mathcal{N} \kappa_i \|H_\epsilon(\theta-c_{i-1})-H_\epsilon(\tilde{\theta}-c_{i-1})\|_{L^2(D)}\\\nonumber
        &\,\,\,\,\,+ \sum_{i=1}^\mathcal{N} \kappa_i \|H_\epsilon(\theta-c_{i})-H_\epsilon(\tilde{\theta}-c_{i})\|_{L^2(D)},\\\nonumber
        &\leq \epsilon^{-1} \sum_{i=1}^\mathcal{N} \kappa_i \|\theta-\tilde{\theta}\|_{L^2(D)},\\\nonumber
        &\leq C(\mathbf{\kappa},\mathcal{N},D) \epsilon^{-1} \|\theta-\tilde{\theta}\|_{L^\infty(D)}. 
    \end{eqnarray}
\end{proof}
For the following consistency result we let
\begin{equation}\label{eq:level-setting}
    \Theta = H_\diamond^\beta(\mathcal{X}), \qquad \mathcal{S}_\beta(M):=\{\theta \in H_\diamond^\beta(\mathcal{X}): \|\theta\|_{H^\beta(\mathcal{X})}\leq M \}.
\end{equation}
We endow $\Theta$ with a prior distribution $\Pi_\theta$ that satisfies Condition \ref{cond:prior} for $\beta>2+d'/2$ such that the covariance kernel associated with the random field is stationary. For simplicity we assume $f(x) = x^\nu$ for some $0<\nu<1$ in Condition \ref{cond:forward-map}. Then we have the following result proved in Appendix \ref{sec:extra-proofs}.


\begin{theorem}\label{thm:level-set-const}
    Suppose Condition \ref{cond:forward-map} is satisfied for $\mathcal{S}_\beta(M)$ as in \eref{eq:level-setting} for $f(x)=Cx^\nu$, $\Phi$ replaced by $\Phi_{n^{-k}}$ for a well-chosen $k$, and where $C$ and $C_{\mathcal{G}}$ are independent of $n$. Let $\gamma_0=\Phi(\theta_0)$ for $\theta_0\in \mathcal{H}\cap \Theta$. Let $\Pi(\cdot|Y)$ be the corresponding sequence of posterior distributions arising for the model \eref{eq:obs} and prior $\Pi=\Phi_{n^{-k}}\Pi_\theta$ as above. Then,
    \begin{equation}
        \|E[\gamma|Y]-\gamma_0\|_{L^2(D)}\rightarrow 0 \quad \textnormal{ in $P_n^{\gamma_0}$-probability}
    \end{equation}
    with rate $n^{-a\nu}$ as $n\rightarrow \infty$ for
    \begin{equation}\label{eq:new-a}
    a = \frac{\eta\delta}{2d\nu \eta + 2\eta\delta + d}.
    \end{equation}
    \end{theorem}

Note that for weak inverse stability estimates, i.e. $\nu$ small, the obtained contraction rate approaches the usual rate \eref{eq:contraction-rate-a}.

\section{Quantitative photoacoustic tomography problem}
\label{section:5}
To test the convergence of the inclusion detection methods, we consider the following test problem in quantitative photoacoustic tomography, see \cite{tarvainen2012,bal2011,kuchment2012}.
The diffusion approximation in QPAT models light transport in a scattering medium according to an elliptic equation
\begin{equation}\label{eq:goveq}
\eqalign{
-\nabla \cdot \mu \nabla u + \gamma u &=0, \text { in } D, \\
u = g, \text { on } \partial D,
}
\end{equation}
where $\mu\in L^2_{\Lambda_\mu}(D)$, $\Lambda_\mu>0$, and $\gamma \in L^2_\Lambda(D)$ are the optical diffusion and absorption parameters, respectively. The prescribed Dirichlet boundary condition $u=g$ defines the source of incoming radiation. It is well-known that \eref{eq:goveq} has a unique solution $u\in H^1(D)$ for each $g\in H^{1/2}(\partial D)$ and for any nonzero source function $h\in H^{-1}(D)$ of \eref{eq:goveq}. Furthermore, we have the estimate
\begin{equation}\label{eq:ellip-estimate}
   \|u\|_{H^1(D)}\leq C(\Lambda_\mu,D)(\|g\|_{H^{1/2}(\partial D)}+\|h\|_{H^{-1}(D)}), 
\end{equation}
see for example \cite[Chapter 6]{evans1998}. QPAT aims to reconstruct the optical parameters given the absorbed optical energy density map $H$, 
which equals the product $\gamma u$ up to some proportionality constant that models the photoacoustic effect. In our simplified approach, we aim to invert the forward map 
\begin{eqnarray}\nonumber 
	\mathcal{G}: \gamma \mapsto H := \gamma u, \quad \mathcal{G}:L^2_\Lambda(D) \rightarrow L^2(D),
\end{eqnarray}
for a fixed $\mu \in L^2_\Lambda(D)$. For smoothness and physical accuracy we assume
\begin{equation}\label{eq:assump-on-f}
    g\in H^{3/2}(\partial\Omega) \quad \textnormal{ and } \quad 0<g_{\mathrm{min}}\leq g \leq g_{\mathrm{max}}.
\end{equation}
This setting allows a simple inverse stability estimate. First we have the following continuity result of $\mathcal{G}$.
\begin{lemma}\label{lemma:forward-reg-qpat}
	Let $H_1:=\gamma_1u_1$ and $H_2:=\gamma_2u_2$ for solutions $u_1$ and $u_2$ of \eref{eq:goveq} corresponding to $\gamma=\gamma_1$,  $\gamma=\gamma_2$ in $L^2_\Lambda(D)$ and $g$ satisfying \eref{eq:assump-on-f}. Then there exists a constant $C$ such that
	\begin{eqnarray}\nonumber 
		\|H_1-H_2\|_{L^2(D)} \leq C \|\gamma_1-\gamma_2\|_{L^2(D)},
	\end{eqnarray}
	where $C$ depends on $\Lambda_\mu$, $D$ and $g_{\mathrm{max}}$.
\end{lemma}
\begin{proof}
    We note that $u_1-u_2$ solves
    \begin{eqnarray}\nonumber 
    -\nabla \cdot \mu \nabla (u_1-u_2) + \gamma_1 (u_1-u_2) &=u_2(\gamma_2-\gamma_1) \text { in } D, \\ \nonumber
    u_1-u_2 = 0  \text { on } \partial D.
    \end{eqnarray}
    Then by \eref{eq:ellip-estimate} and the maximum principle \cite[Theorem 8.1]{gilbarg83}
    \begin{eqnarray}\nonumber
        \|u_1-u_2\|_{H^1(D)}\leq \|u_2(\gamma_2-\gamma_1)\|_{H^{-1}(D)} \leq f_{\mathrm{max}}\|\gamma_1-\gamma_2\|_{L^2(D)}.
    \end{eqnarray}
    Since $H_1-H_2=\gamma_1(u_1-u_2)+(\gamma_1-\gamma_2)u_2$ we have
    \begin{eqnarray}\nonumber
        \|H_1-H_2\|_{L^2(D)}&\leq \|\gamma_1(u_1-u_2)\|_{L^2(D)} + \|(\gamma_1-\gamma_2)u_2\|_{L^2(D)},\\\nonumber
        &\leq f_{\mathrm{max}}(1+M) \|\gamma_1-\gamma_2\|_{L^2(D)}.
    \end{eqnarray}
\end{proof}

\begin{lemma}\label{lemma:stability-estimate-qpat}
    Under the same assumptions of Lemma \ref{lemma:forward-reg-qpat}, there exists a constant $C>0$ such that
	\begin{equation}\label{eq:inverse-stability-estimate-qpat}
		\|\gamma_1-\gamma_2\|_{L^2(D)} \leq C \|H_1-H_2\|_{L^2(D)}.
	\end{equation}
\end{lemma}
\begin{proof}
    See also \cite[Theorem 3.1]{bal2011} and \cite[Theorem 1.2]{choulli2021}. Note 
    $u_1-u_2\in H^1_0(D)$ solves
    \begin{eqnarray}\nonumber 
    -\nabla \cdot \mu \nabla (u_1-u_2) &= H_2-H_1 \text { in } D, \\\nonumber 
    u_1-u_2 = 0 \text { on } \partial D,
    \end{eqnarray}
    hence by elliptic regularity 
    \begin{equation}\label{eq:ellipticregest}
        \|u_1-u_2\|_{L^2(D)}\leq C(\Lambda_\mu,D)\|H_1-H_2\|_{L^2(D)}.
    \end{equation}
    Note by the trace theorem, see \cite{lions1972}, for $g$ as in \eref{eq:assump-on-f} there exists $v\in H^2(\Omega)$ such that $u_2-v\in H^1_0(\Omega)$. By a Sobolev embedding $v\in C^{0,\alpha_0}(\overline{D})$ for some $\alpha_0>0$ depending on $d=2,3$. Theorem 8.29 and the remark hereafter in \cite{gilbarg83} states that $u_2\in C^\alpha(\overline{D})$ for some $\alpha=\alpha(d,\Lambda_\mu,\Lambda,D,\alpha_0)>0$ and that
    \begin{eqnarray}\nonumber 
        \|u_2\|_{C^\alpha(\overline{D})} \leq U_1(\sup_{x\in D}|u(x)|+U_2)=:U,
    \end{eqnarray}
    where $U_1=M_1(d,\Lambda_\mu,\Lambda,D,\alpha_0)>0$ and $U_2=U_2(D,g)$. By the maximum principle \cite[Theorem 8.1]{gilbarg83} we can collect the right-hand side to one constant $U=U(U_1,U_2,g_{\mathrm{max}})>0$. Now using the argument in \cite[Lemma 12]{choulli2021}, which in return uses the Harnack inequality \cite[Corollary 8.21]{gilbarg83} we conclude
    \begin{equation}\label{eq:lower-bound-u2}
        u_2 \geq m,
    \end{equation}
    where $m=m(d,\Lambda_\mu,\Lambda,D,U,\alpha,g_{\mathrm{min}})$ is a constant. Note
    \begin{eqnarray}\nonumber
        \gamma_1-\gamma_2 &= \gamma_1(1-\frac{u_1}{u_2})+\frac{1}{u_2}(\gamma_1u_1-\gamma_2u_2),\\\nonumber
        &= \frac{\gamma_1}{u_2}(u_2-u_1) + \frac{1}{u_2}(H_1-H_2).
    \end{eqnarray}
    Combining this with \eref{eq:ellipticregest} and \eref{eq:lower-bound-u2} we have
    \begin{eqnarray}\nonumber
        \|\gamma_1-\gamma_2\|_{L^2(D)}&\leq C(m,\Lambda,\Lambda_\mu,D) \|H_1-H_2\|_{L^2(D)}. 
    \end{eqnarray}
\end{proof}
We note that $\mathcal{G}$ satisfies Condition \ref{cond:forward-map} for $\eta=1$ and $f(x)=x$. We also note that $\mathcal{Y}=L^2(D)$ is a separable Hilbert space with an orthonormal basis consisting of the eigenfunctions of the Dirichlet Laplacian on $D$. We conclude that this problem is suitable as a test problem, and that Theorem \ref{thm:star}
 and \ref{thm:level-set-const} apply. In Section \ref{sec:conclusions} we discuss other suitable inverse problems.
\section{Numerical results}
\label{section:6}

We discuss our numerical tests in detecting inclusions for the QPAT tomography problem using the pCN algorithm of Section \ref{sec:mcmc} and the parametrizations of Section \ref{section:4}. For simplicity we assume $D=B(0,1)$, the two-dimensional unit disk. 

\subsection{Observation model}
As an approximation to the continuous observation model \eref{eq:obs} for the numerical experiments we consider observing
\begin{equation}\label{eq:observation-numerics}
    Y_k = \langle \mathcal{G}(\gamma), e_k \rangle_{L^2(D)} + \varepsilon \xi_k, \qquad k=1,\ldots,N_{d}
\end{equation}
where $\{e_k\}_{k=1}^\infty$ is an orthonormal basis of $L^2(D)$ consisting of the eigenfunctions of the Dirichlet Laplacian on $D$ and $N_{d}\in \mathbb{N}$ is a suitable number. This observation $\mathbf{Y}=\{Y_k\}_{k=1}^{N_d}$ is the sequence of coefficients of the projection of $Y$ from \eref{eq:obs} to the span of $\{e_k\}_{k=1}^{N_{d}}$. As $N_{d}\rightarrow \infty$ observing $\mathbf{Y}$ is equivalent to observing $Y$, see for example \cite[Theorem 26]{abraham2019}. Besides being a convenient approximation, this model has numerical relevance: there exists closed-form reconstruction formulas for  $\langle \mathcal{G}(\gamma), e_k \rangle_{L^2(D)}$ in the first part of the photoacoustic problem, see \cite{kunyansky2007,agranovsky2007}.
The likelihood function then takes the form
\begin{eqnarray}\nonumber p^\gamma_\varepsilon(\mathbf{Y}):=\exp\left(-\frac{1}{\varepsilon^2}\sum_{k=1}^{N_{d}} (Y_k-\langle \mathcal{G}(\gamma), e_k \rangle_{L^2(D)})^2 \right).\end{eqnarray}
\subsection{Approximation of the forward map}\label{sec:approx-forward}
We approximate the forward map using the Galerkin finite element method (FEM) with piecewise linear basis functions $\{\psi_k\}_{k=1}^{N_m}$ over a triangular mesh of $N_m$ vertices and $N_e$ elements, see \cite{tarvainen2012,hänninen2018}. 
When $\gamma\in L^2_\Lambda(D)$ is discontinuous and continuous, we approximate it by
\begin{eqnarray}\nonumber \tilde{\gamma}_{N_e}=\sum_{k=1}^{N_e} \tilde{\gamma}_k \mathds{1}_{E_k}, \quad \textnormal{and} \quad \bar{\gamma}_{N_m}=\sum_{k=1}^{N_m} \bar{\gamma}_k \psi_k,\end{eqnarray}
respectively. Here $E_k$ denotes the $k$'th element of the triangular mesh. That gives us two approximations of the forward map:
\begin{eqnarray}\nonumber \tilde{\mathcal{G}}_{N_e}(\gamma) := \tilde{\gamma}_{N_e}\tilde{u} \quad \textnormal{and} \quad \bar{\mathcal{G}}_{N_m}(\gamma) := \bar{\gamma}_{N_m}\bar{u},\end{eqnarray}
where $\tilde{u}$ is the FEM solution corresponding to $\tilde{\gamma}_{N_e}$ and $\bar{u}_{N_m}$ is the FEM solution corresponding to $\bar{\gamma}$. For the smooth level set parametrization we use $\bar{\mathcal{G}}_{N_m}$ with $N_m = 12708$ nodes, while for the star-shaped set parametrization we use $\tilde{\mathcal{G}}_{N_e}$ with $N_e = 25054$ elements.

We compute $\{e_k\}_{k=1}^{N_{d}}$ by solving the generalized eigenvalue problem arising from the FEM formulation of the Dirichlet eigenvalue problem with the \textsc{Matlab} function $\texttt{sptarn}$. Then $\langle \mathcal{G}(\gamma),e_k\rangle_{L^2(D)}$ is approximated using the mass matrix for $k=1,\ldots,N_d$ with $N_d=N_{\mathrm{freq}}(N_{\mathrm{freq}}+1)$ and $N_{\mathrm{freq}} = 13$. 

\subsection{Phantom, noise and data}\label{sec:phantom}
The phantom we seek to recover consists of two inclusions:
\begin{eqnarray}\nonumber \gamma_0 = \kappa_1 + \kappa_2\mathds{1}_{A_1} + \kappa_3\mathds{1}_{A_2},\end{eqnarray}
where $(\kappa_1,\kappa_2,\kappa_3)=(0.1, 0.4, 0.2)$ and $A_1,A_2$ are two star-shaped sets  described by their boundaries:
\begin{eqnarray}\nonumber \fl \partial A_1 = (-0.4,0.4) + \{0.18(\cos(\vartheta)+0.65\cos(2\vartheta), 1.5\sin(\vartheta)), 0\leq \vartheta \leq 2\pi\},\end{eqnarray}
\begin{eqnarray}\nonumber \fl \partial A_2 = (0.4,-0.4) + \{\varphi(\vartheta)(\cos(\vartheta),\sin(\vartheta))\},\end{eqnarray}
where $\varphi(\vartheta) = 0.12\sqrt{0.8+0.8(\cos(4\vartheta)-1)^2}$, see Figure \ref{fig:optical-parameters}. We compute and fix the optical diffusion parameter to $\mu = \frac{1}{2}\frac{1}{\gamma_0+\mu_s(1-0.8)}$ following \cite{hänninen2018}. Here the scattering parameter $\mu_s$ equals $100\gamma_0$ smoothed with a Gaussian smoothing kernel of standard deviation $15$ using the \textsc{Matlab} function \texttt{imgaussfilt}.
\begin{figure}[tbp!]
    \centering
    \includegraphics[width=0.9\textwidth]{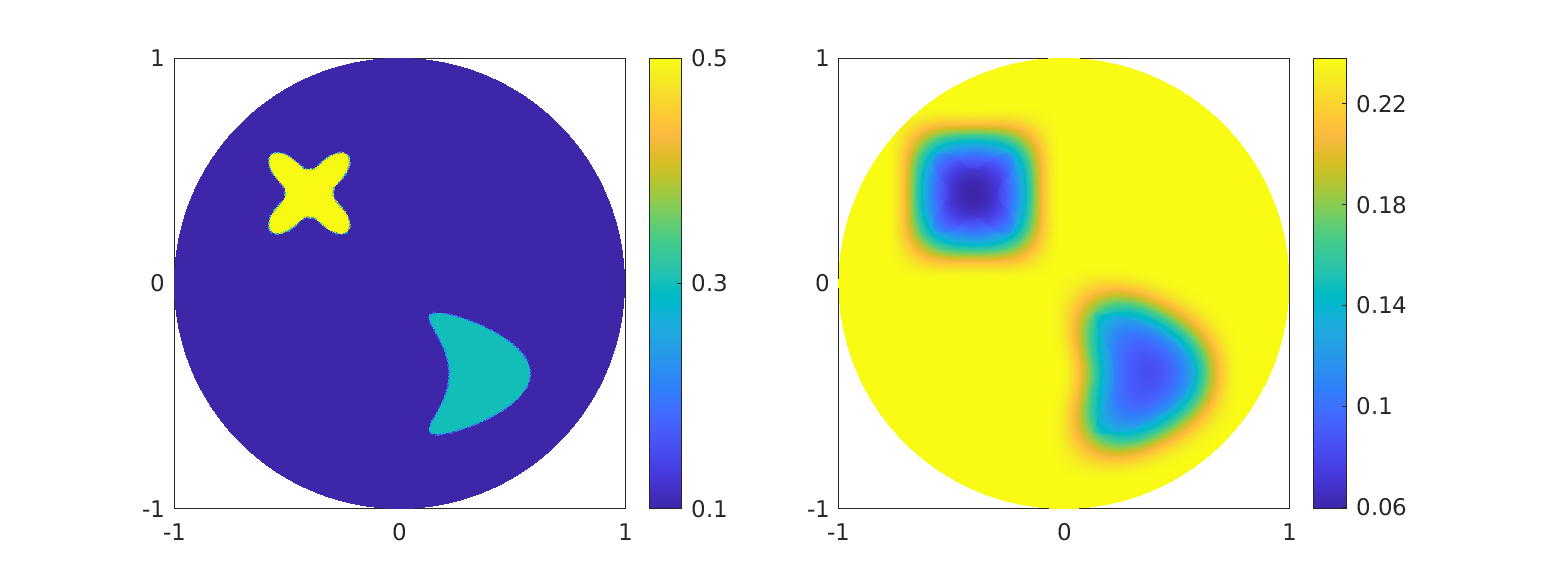}
    \caption{Simulated absorption $\gamma_0$ (left image) and diffusion $\mu$ (right image) distributions.}
    \label{fig:optical-parameters}
\end{figure}
We choose an illumination $g$ that is smooth and positive on $\partial D$ defined by
\begin{eqnarray}\nonumber g(x) = w_{m_1,s_1}(x) + w_{m_2,s_3}(x) + w_{m_2,s_3}(x),\end{eqnarray}
where
\begin{eqnarray}\nonumber w_{m,s}(x)=s \exp\left(-2\|x-m\|^2\right)\end{eqnarray}
and $m_1 = 0.5(\sqrt{2},\sqrt{2})$, $m_2= 0.5(-\sqrt{2},\sqrt{2})$, $m_3=-m_1$, $s_1 = 10$, $s_2=2$ and $s_3=5$. This is a superposition of three Gaussians, which illuminates the target well.
\begin{figure}[tbp!]
    \centering
    \includegraphics[width=0.9\textwidth]{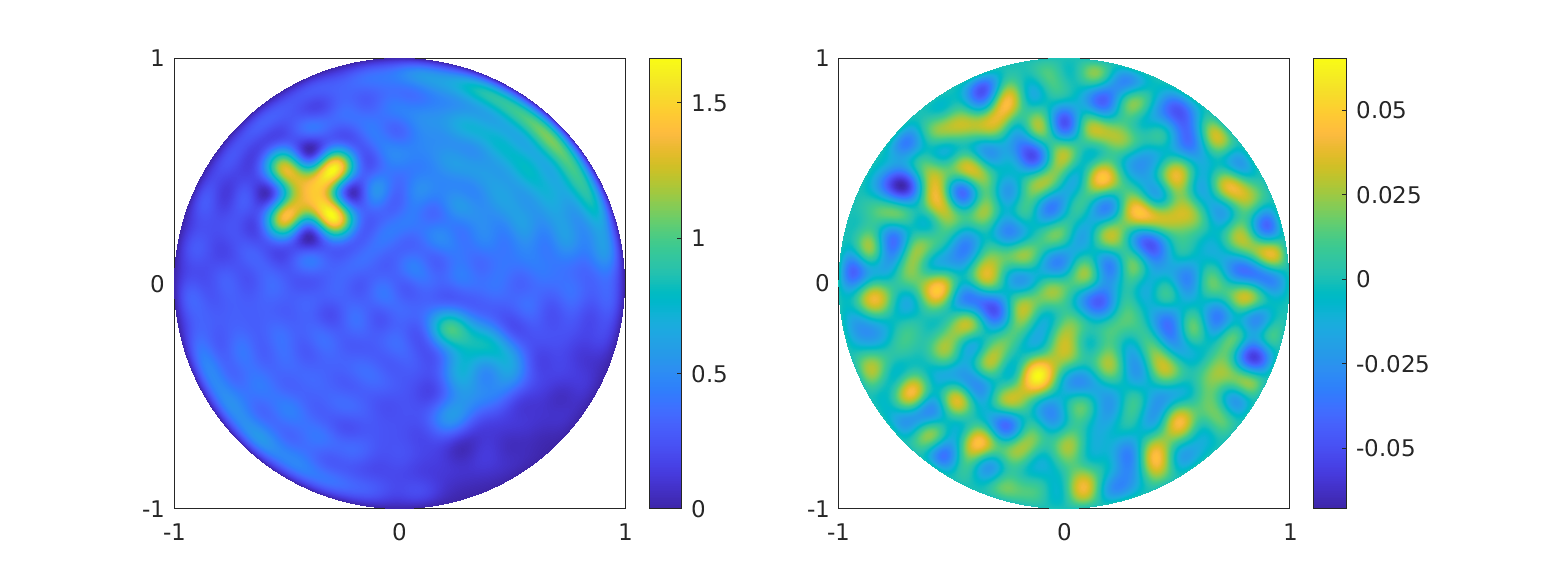}
    \caption{Projection of absorbed optical energy density $H$ corresponding to phantom of Section \ref{sec:phantom} (left image) and of the white noise expansion (right image) projected onto the span of $\{e_k\}_{k=1}^{N_d}$.} 
    \label{fig:data}
\end{figure}
We simulate data $\mathbf{Y}$ as in \eref{eq:observation-numerics} by computing $\tilde{\mathcal{G}}_{N_{e_0}}(\gamma_0)$ on a fine mesh of $N_{e_0} = 75624$ elements and $N_{m_0}=38127$ nodes. The corresponding projection can be seen in Figure \ref{fig:data}. We choose $\varepsilon>0$ such that the relative error
\begin{eqnarray}\nonumber \textnormal{relative error} = \frac{\varepsilon \sqrt{\sum_{k=1}^{N_d} \xi_k^2}}{\sqrt{\sum_{k=1}^{N_d} \langle \mathcal{G}(\gamma_0), e_k \rangle_{L^2(D)}^2}}\end{eqnarray}
is in the range $(1, 2, 4, 8, 16)\cdot 10^{-2}$.
 See Figure \ref{fig:data} for a realization of the white noise expansion \eref{eq:white-noise} projected to the $N_d$ first orthonormal vectors $\{e_k\}_{k=1}^{N_d}$ and scaled so that it accounts for $4\%$ relative noise.\\ 

To estimate the approximation error, we compute the vector 
\begin{eqnarray}\nonumber V_j = [\langle\tilde{\mathcal{G}}_{N_0}(\gamma_j), e_k\rangle_{L^2(D)}-\langle\tilde{\mathcal{G}}_{N_e}(\gamma_j), e_k \rangle_{L^2(D)}]_{k=1}^{N_d}\end{eqnarray}
for $\gamma_j$, $j=1,\ldots,200$, samples of the prior for the level set parametrization introduced in Section \ref{sec:level-set-prior} below. We then compute $\varepsilon_{\mathrm{level}} = \sqrt{\frac{\mathrm{tr}(C)}{N_d}}$, where $\mathrm{tr}(C)$ is the trace of the sample covariance matrix $C$ of the vectors $V_j$. For this choice $N(0,\varepsilon_{\mathrm{level}}^2 I)$ minimizes the Kullback-Leibler distance to $N(0,C)$, see \cite{cotter2013}. We compute $\varepsilon_{\mathrm{star}}$ in the same way using $\bar{\mathcal{G}}_{N_{m_0}}$, $\bar{\mathcal{G}}_{N_m}$ and samples of the prior for the star-shaped set parametrization in Section \ref{sec:level-set-prior}.

\subsection{Choice of prior}
\subsubsection{Star-shaped sets}\label{sec:star-shape-prior}
To mirror the theoretical results of Theorem \ref{thm:star} for the phantom above, we consider a product distribution in $H^\beta(\mathbb{T})\times H^\beta(\mathbb{T})$.
To this end, consider the usual $L^2([0,2\pi])$ real orthonormal basis of trigonometric functions $\{\phi_\ell\}_{\ell\in \mathbb{Z}}$, i.e. $\phi_1(x) = \cos(2\pi x)$ and $\phi_{-1}(x) = \sin(2\pi x)$. 
Consider the Karhunen-Loeve expansion
\begin{equation}\label{eq:random-series}
    \theta_i = \bar{\theta} + \sum_{\ell \in \mathbb{Z}} g_{\ell,i} w_{\ell}\phi_\ell, \qquad g_{\ell,1},g_{\ell,2} \stackrel{i.i.d}{\sim} N(0,1),
\end{equation}
for $i=1,2$ with $w_\ell = q(\tau^2+|\ell|^2)^{-\delta/2}$ for $\delta>1/2$, $\tau\in \mathbb{R}$, $q>0$ and some constant $\bar{\theta}\in \mathbb{R}$. Note $\theta_i$ has a Laplace-type covariance operator, and \eref{eq:random-series} can be interpreted as the solution of a stochastic PDE \cite{lindgren2011}. Then $\theta_1,\theta_2\in H^\beta(\mathbb{T})$ almost surely, see \cite{dashti2017}. According to Theorem I.23 in \cite{ghosal2017} and the definition of Sobolev spaces \cite[Section 4.3]{taylor2011},   $\mathcal{H}=H^{\delta}(\mathbb{T})$ with equivalent norms, i.e. the prior distribution of \eref{eq:random-series} satisfies Condition \ref{cond:prior}.
We take as $\Pi$ the distribution of
\begin{equation}\label{eq:push-forward-prior}
    \gamma = \Phi(\theta_1,\theta_2) = \kappa_1 + \kappa_2 \mathds{1}_{A(x_1,\theta_1)} + \kappa_3 \mathds{1}_{A(x_2,\theta_2)},
\end{equation}
for $(\kappa_1,\kappa_2,\kappa_3)=(0.1,0.2,0.4)$, $x_1=(0.37,-0.43)$ and $x_2 = (-0.44,0.36)$.
 In practice, we compute \eref{eq:random-series} truncated at $|\ell|\leq N=12$. We do not rescale as the theoretical estimates demand. Instead, we handpick a suitable $q$ for each noise level. We use \texttt{inpoly} \cite{kepner2020} to efficiently project $\gamma$ to $\{\mathds{1}_{E_k}\}_{k=1}^{N_e}$.
 We refer to Figure \ref{fig:prior-samples} for an example of a sample from this prior.

\subsubsection{Level sets}\label{sec:level-set-prior}
For the level set parametrization, we consider a prior distribution in $H^\beta(\tilde{\mathbb{T}}^2)$. Here $\tilde{\mathbb{T}}^2$ is the torus corresponding to the square $[-m,m]^2$, where we choose $m=1.1$, since it is recommended in for example \cite{khristenko2019} to embed $D$ in a larger domain to avoid boundary effects. 
Here, we consider the usual $L^2([-m,m]^2)$ real orthonormal basis of trigonometric functions $\{\phi_\ell\}_{\ell\in \mathbb{Z}^2}$. We let
\begin{equation}\label{eq:KL-level}
    \theta = \sum_{\ell\in \mathbb{Z}^2} g_\ell w_\ell \phi_\ell, \qquad g_\ell \stackrel{i.i.d}{\sim} N(0,1),
\end{equation}
with $w_\ell = q(\tau^2+|\ell|^2)^{-\delta/2}$ for $\delta > 1$, $\tau \in \mathbb{R}$ and $q>0$.
Similar to above, the series exists almost surely as an element in $H^\beta(\tilde{\mathbb{T}}^2)$, see \cite{dashti2017}. The corresponding RKHS is $\mathcal{H}=H^{\delta}(\tilde{\mathbb{T}}^2)$, see \cite{ghosal2017}. We choose $\mathcal{X}=D$ and consider the linear, bounded and surjective restriction $r: H^\beta(\tilde{\mathbb{T}}^2)\rightarrow H^\beta(D)$, see \cite[Section 4.4]{taylor2011}. Then $r(\theta)$ is a Gaussian random element in $H^\beta(D)$, and its RKHS is $r(\mathcal{H})=H^\delta(D)$, see \cite[Exercise 2.6.5]{gine2016}. We take as $\Pi$ the distribution of 
\begin{equation}\label{eq:push-forward-level}
    \gamma = \Phi_{\epsilon}(\theta) = \sum_{i=1}^3\kappa_i[H_\epsilon(\theta-c_{i-1})-H_\epsilon(\theta-c_i)]
\end{equation}
for $(\kappa_1,\kappa_2,\kappa_3)=(0.3,0.1,0.5)$ and $(c_1,c_2) = (-1,1)$. In practice, we truncate \eref{eq:KL-level} at $\max(|\ell_1|,|\ell_2|)\leq 4$. Also, we hand-pick $\epsilon>0$ and $q>0$ for each noise level.
See Figure \ref{fig:prior-samples} for a sample of this prior.

\begin{figure}[tbp!]
    \centering
    \includegraphics[width=\textwidth]{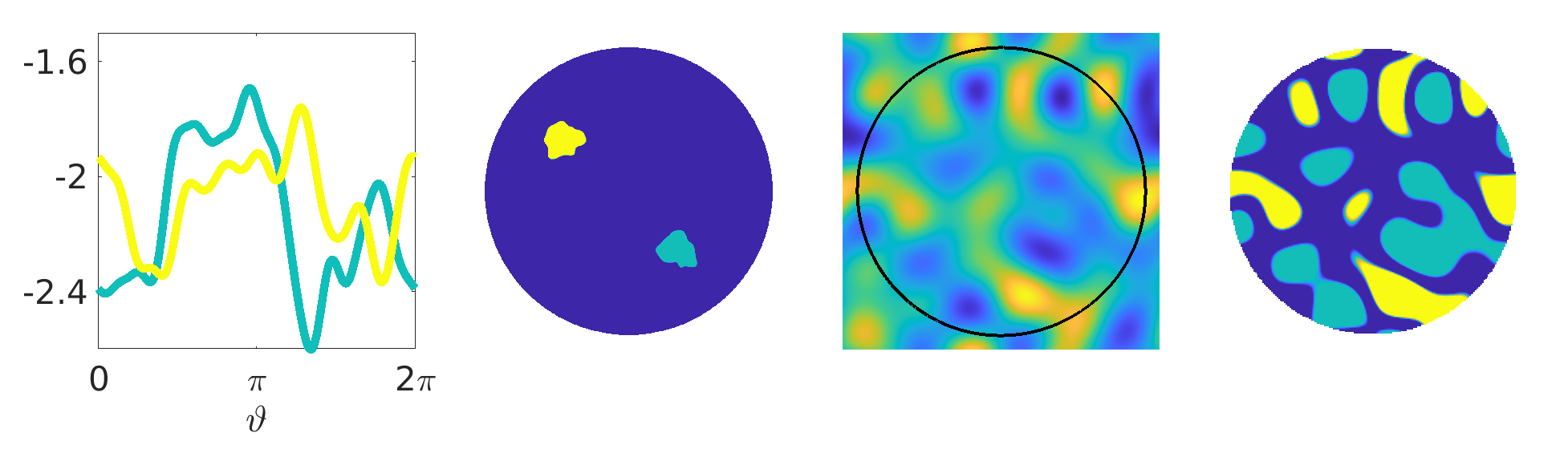}
    \caption{Samples from the star-shaped and level set priors. From left to right: Sample of $\theta_1$ and $\theta_2$ in \eref{eq:random-series} for $|\ell|\leq 12$, $\delta=2.5$, $\tau =4$, $q=10^{3/2}/5$ and $\bar{\theta}=-2$ (first image). Sample of $\Pi$ corresponding to $\Phi(\theta_1,\theta_2)$ in \eref{eq:push-forward-prior} (second image). Sample of $\theta$ in \eref{eq:KL-level} for $\max(|\ell_1|,|\ell_2|)\leq 4$, $\delta = 1.2$, $\tau = 10$, $q=5$ (third image). Sample of $\Pi$ corresponding to $\Phi_\epsilon(\theta)$ in \eref{eq:push-forward-level} for $\epsilon=0.1$ (fourth image).}
    \label{fig:prior-samples}
\end{figure}

\subsection{Results}
In this section we present the numerical results using the star-shaped and level set parametrizations in different noise regimes. 
We use the algorithm described in Section \ref{sec:mcmc} with the pCN method implemented with an adaptive stepsize targeting a 30\% acceptance rate. The initial stepsize is denoted by $b$. For an example of an implementation of this sampling method, we refer to the \textsc{Python} package \textsc{CUQIpy}, see \cite{riis2023}. 
For the star-shaped set parametrization, we choose the following prior and algorithm parameters in the order $(16, 8, 4, 2, 1)\cdot 10^{-2}$ of the relative noise levels: $b=(0.1, 0.045, 0.035, 0.025, 0.015)$, $q=10^{3/2}\cdot(7/20, 6/20, 5/20, 4/20, 3/20)$, $\delta = 2.5$, $\bar{\theta}=-2$, $\tau=4$ and $\theta^{(0)} = (1,1)$ corresponding to inclusions of constant radius. In the same order, we choose for the level set parametrization the following prior and sampling parameters: $b = (0.05,0.01,0.006,0.003,0.002)$, $q=5 \cdot(5/2, 2, 3/2, 1, 3/4)$, $\delta = 1.2$, $\tau=10$ and $\theta^{(0)} = 2\phi_{(0,-1)} \propto \sin(2\pi/2.2 y)$.\\

 For the star-shaped set parametrization, we obtain $K = 10^6$ samples after a burn-in of $5\cdot 10^5$, whereas for the level set parametrization, we take $K=10^6$ after $1.2\cdot 10^6$ samples as burn-in. We find this choice suitable, since the truncation in Section \ref{sec:star-shape-prior} leaves us with a higher dimensional sampling problem in the level set case. We base our posterior mean approximations on Monte Carlo estimates using $10^2$ equally spaced samples of the chain.\\


In Figure \ref{fig:recon1}, we see the posterior mean of arising from the star-shaped set parametrization and observations with different noise levels. The posterior mean approximates the ground truth well for all noise levels. Note that 
the posterior mean varies only slightly for each noise level and is approximately piecewise constant. This indicates little posterior variance. 
This is due to a small noise level and the fast contraction rate that this inverse problem provides by virtue of \eref{eq:inverse-stability-estimate-qpat}. 
The estimates are not exact, but note that the exact data is not available due to projection and discretization. Taking $N_d$ large improves the data but also causes the likelihood function to attain larger values. This, in return, requires a smaller step size $b$. This means there is a computational trade-off between $N_d$ and $b$. Even for 16\% relative noise, the reconstruction is fairly good, and the variance of the posterior samples is visibly larger. It is a strength of this method that it is robust for large noise levels. The mixing of the sample chains in the trace plots in Figure \ref{fig:sample-chains} indicates that the sampling algorithm is performing well. The convergence of the posterior mean is also evident in $L^2$-distance as computed numerically, see Figure \ref{fig:reg3}. This rate does not match the theoretical; but this is too much to expect for the observation \eref{eq:observation-numerics}, as this does not match the continuous observation \eref{eq:obs} for which the rate is proved. Note we do not numerically scale the priors as the theoretical results require.\\ 
\begin{figure}[tbp!]
    \centering
    \includegraphics[width=0.7\textwidth]{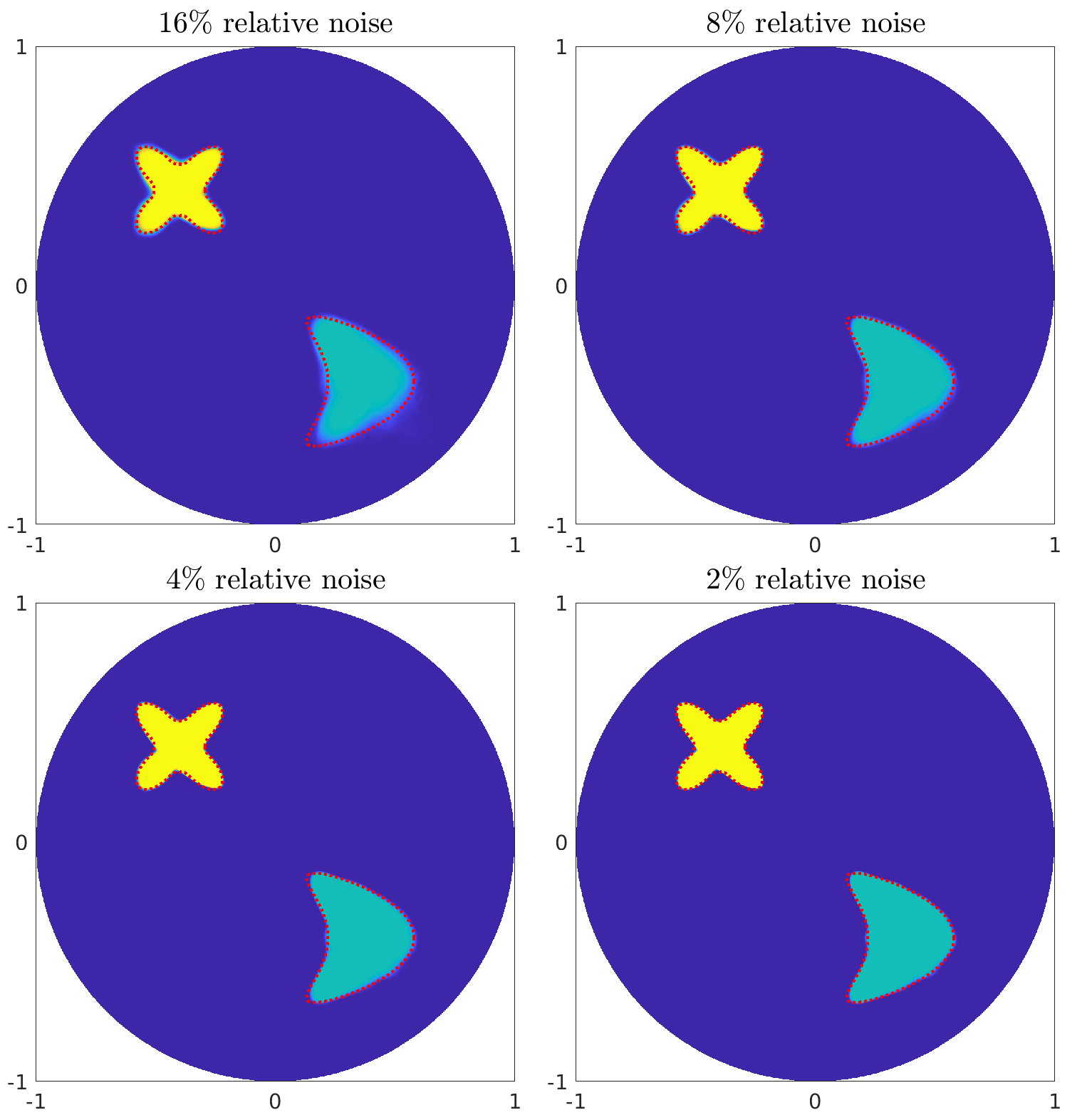}
    \caption{Posterior mean estimates of the absorption parameter using the star-shaped set parametrization in different noise regimes. The dotted red line indicates the location of $\gamma_0$.}
    \label{fig:recon1}
\end{figure}

Figure \ref{fig:recon2} suggests that the posterior mean converges as the noise level goes to zero, as is also evident from its $L^2$-loss in Figure \ref{fig:reg3}. Note that the reconstructions are continuous, not only because we take an average, but also because we use a continuous level set parametrization. Here, the sampling is initialized at $\theta^{(0)} = 2\phi_{(0,-1)}$, since this guess captures some of the low frequency information of possible $\theta_0$ that can give rise to $\gamma_0$. We report that chains with small step-size and the natural starting guess $\theta^{(0)}=0$ often get stuck in local minima due to the number of levels in \eref{eq:push-forward-level} and due to the fact that the pCN method does not require the gradient of either the parametrization or the forward map. \\


The sample diagnostics of Figure \ref{fig:sample-chains} indicate that sampling is  harder for the level set parametrization compared to the star-shaped set parameterization. This is hard for at least two likely reasons: the first is due to the large number of coefficients $\theta_\ell$, $\max(|\ell_1|,|\ell_2|)\leq 4$. This was also noted in \cite{dunlop2016}. The second likely reason is that $\theta \mapsto \Phi_\epsilon(\theta)$ is not injective for any $\epsilon\geq 0$. Therefore, the prior could be multi-modal, and this can lead to correlated samples in the Markov chain. Other work suggests that the pCN method shows an underwhelming performance when applied to a correlated and multi-modal posterior, see \cite{cui2016} which also provides a gradient-based remedy. The level set method has found success in optimization-based approaches, in for example \cite{santosa1995}, where a descent step is taken in each iteration of an iterative algorithm. A Bayesian \emph{maximum a posteriori} approach \cite{reese2021} has also been shown to find success for a smoothened level set. We expect that using gradient information in gradient-based MCMC methods would improve the performance significantly. A benefit of the level set parametrization is that we do not need to know \textit{a priori} the number of inclusions as in the case of the star-shaped set parametrization. One could also combine the two methods as in \cite{afkham2023}. Note in Figure \ref{fig:reg3} that, for both parametrizations, the posterior mean is stable to different noise realizations. This mirrors the convergence in probability we expect from Theorem \ref{thm:star} and \ref{thm:level-set-const}.
\begin{figure}[tbp!]
    \centering
    \includegraphics[width=0.7\textwidth]{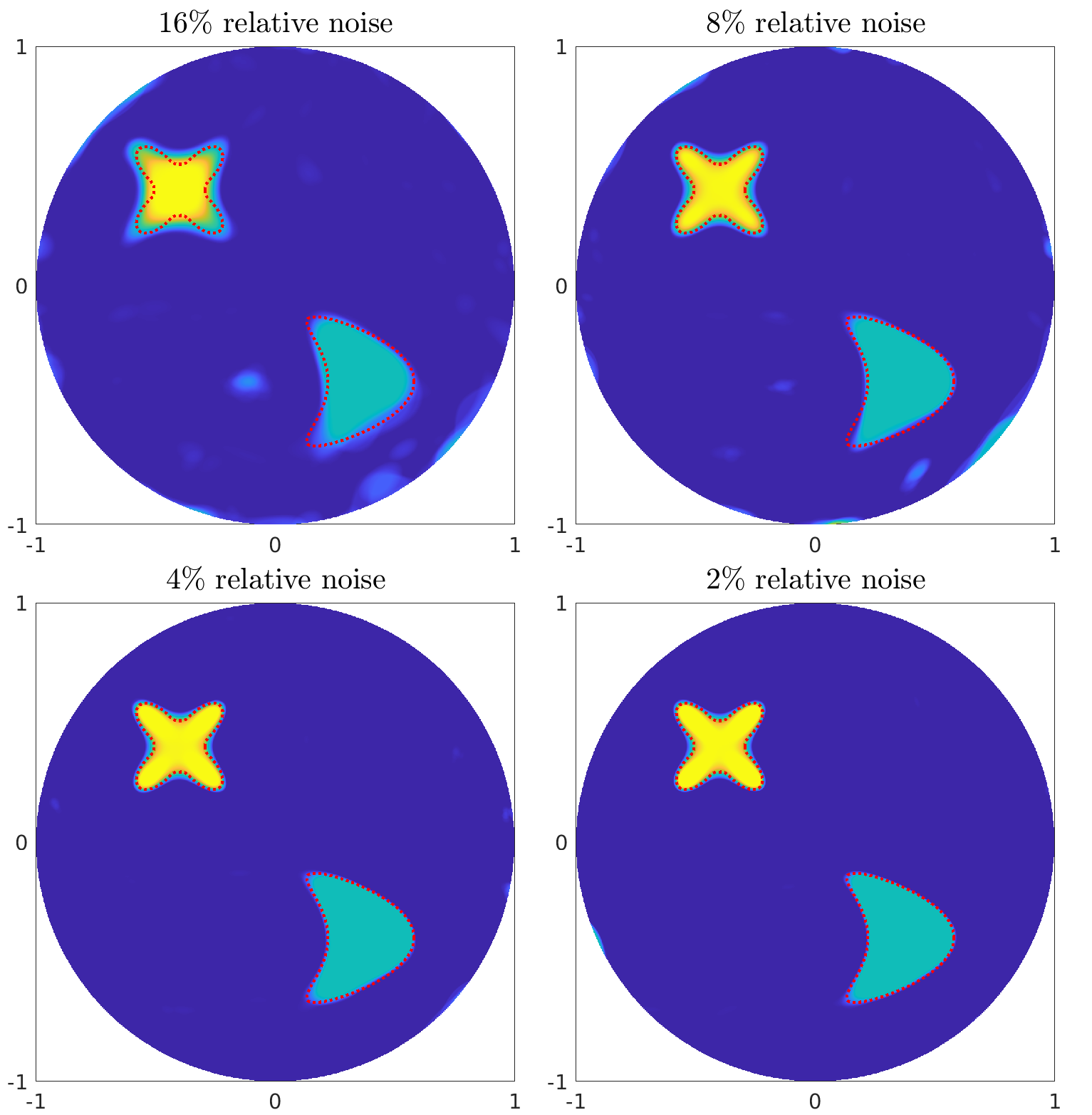}
    \caption{Posterior mean estimates of the absorption parameter using the level set parametrization in different noise regimes. The dotted red line indicates the location of $\gamma_0$.}
    \label{fig:recon2}
\end{figure}

\begin{figure}[tbp!]
\centering
\makebox[\textwidth][c]{
  \subcaptionbox{}{\includegraphics[width=2.1in]{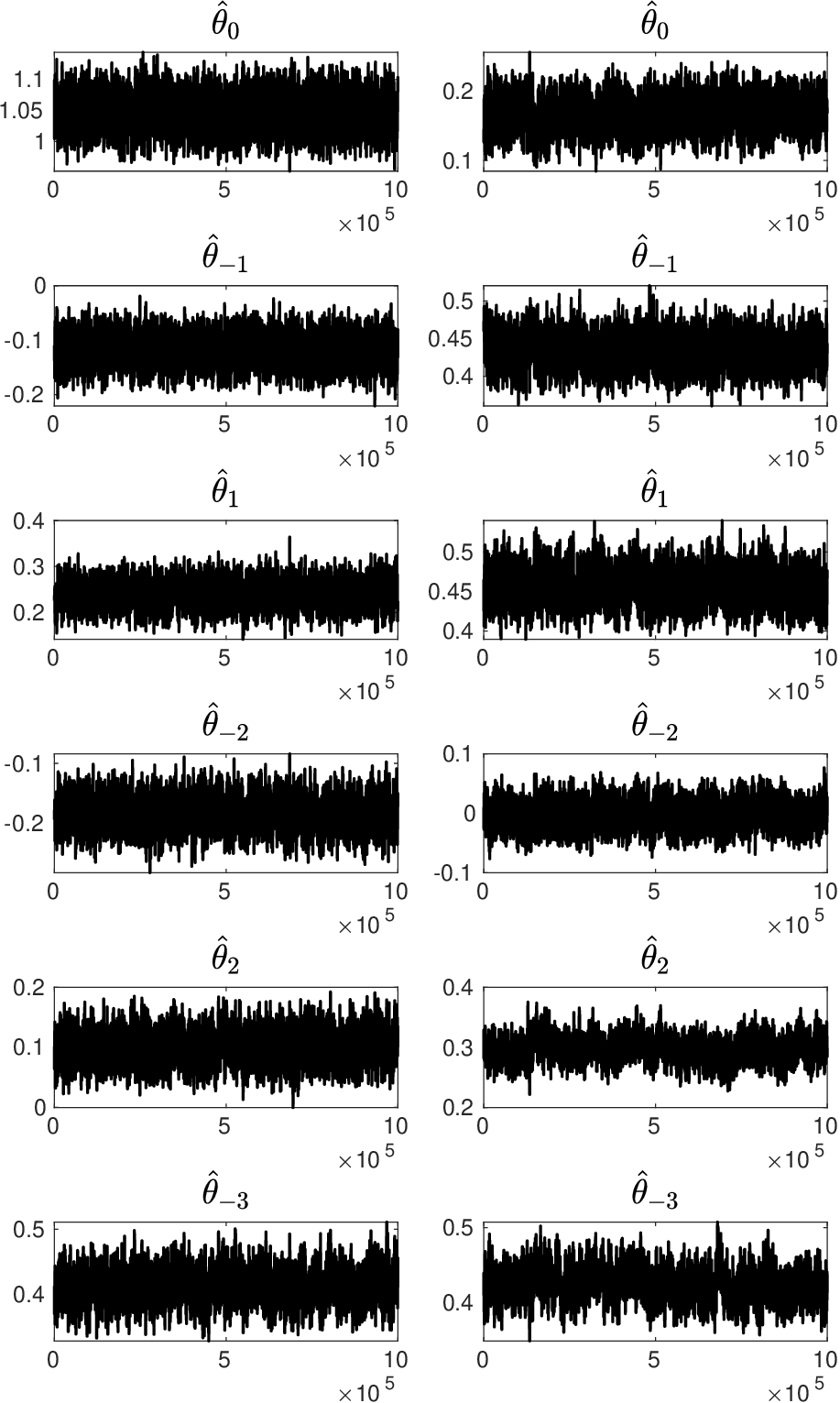}}\hspace{2em} \label{fig:partA}
  \subcaptionbox{}{\includegraphics[width=2.1in]{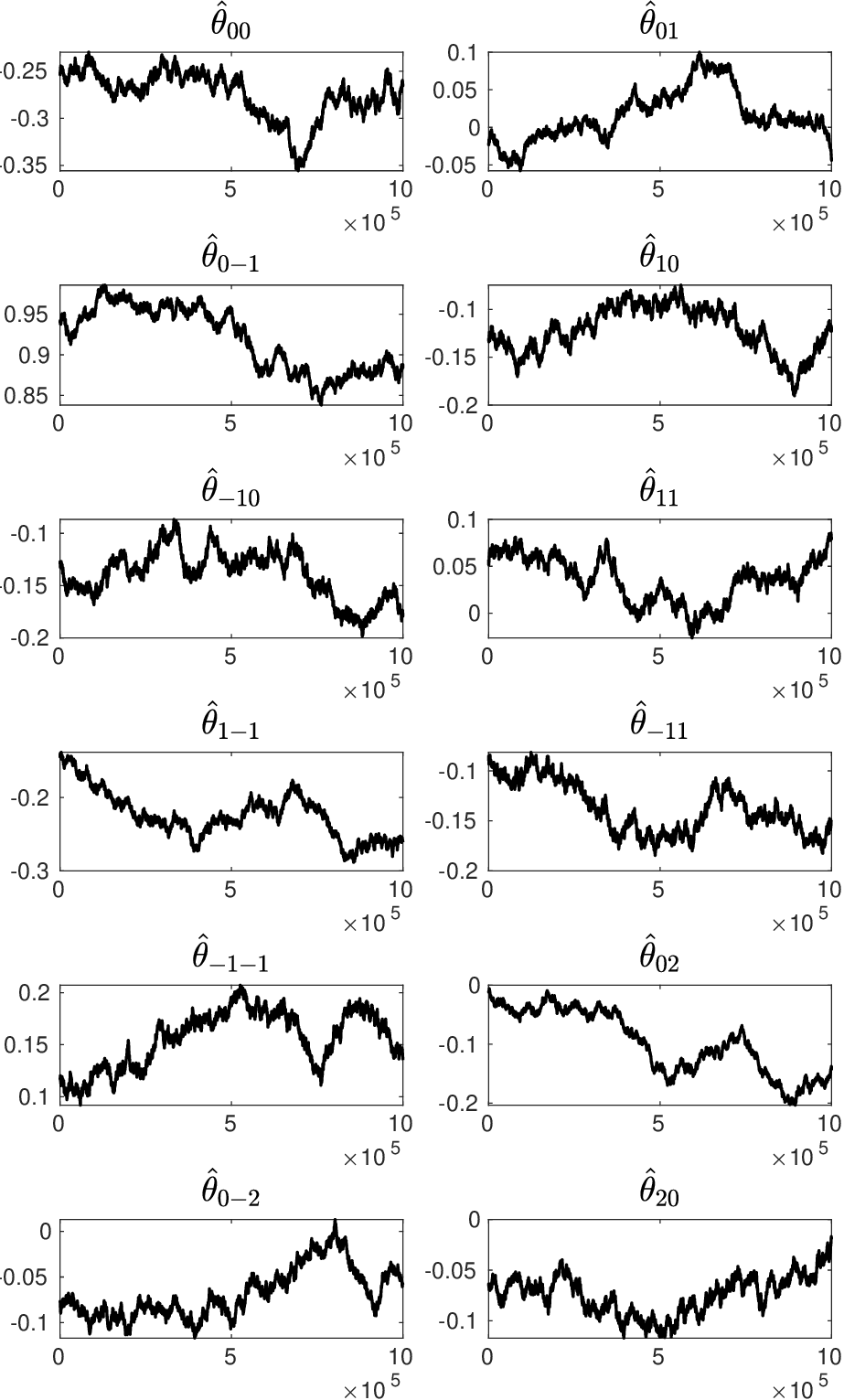}}\label{fig:partB}
  }
  \caption{Plot (\textsc{a}) shows trace plots of the first 6 Fourier coefficients of samples $\theta_1$ (left) and $\theta_2$ (right) from the posterior for the star-shaped set parametrization with observations subject to 4$\%$ relative noise. Plot (\textsc{b}) shows trace plots of the first 12 Fourier coefficients of samples $\theta$ from the posterior for the level set parametrization with observations subject to 4$\%$ relative noise.}
  \label{fig:sample-chains}
\end{figure}

\begin{figure}[tbp!]
        \centering
	\includegraphics[width = 0.58\textwidth]{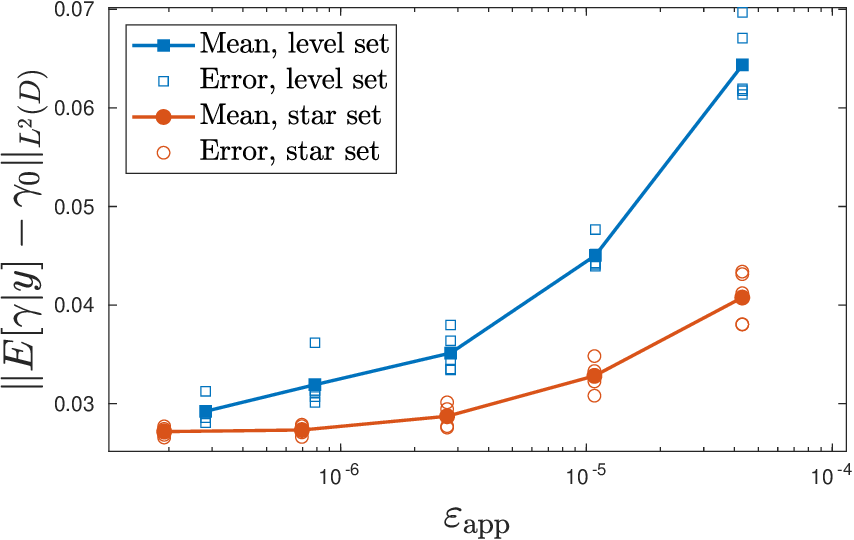}
	\caption{$L^2$-error of 5 realized posterior means for each noise level $\varepsilon_{\mathrm{app}}$ and both parametrizations. The solid markers represent the mean of the 5 error estimates.}
	\label{fig:reg3}
\end{figure}

\section{Conclusions}\label{sec:conclusions}
In this paper, we provide and investigate a Bayesian approach to consistent inclusion detection in nonlinear inverse problems. The posterior consistency analysis is performed under general conditions of Hölder continuity of the parametrization and conditional well-posedness of the inverse problem. Furthermore, it gives an explicit rate. We showcase the convergence of the posterior mean in a small noise limit for a photoacoustic problem, where we note that the star-shaped set parametrization outperforms the level set parametrization. We highlight that Theorem \ref{thm:star1} and \ref{thm:level-set-const} hold for any forward map satisfying Condition \ref{cond:forward-map} and can be applied to other parametrizations. A different parametrization could for example arise in the related problem of crack detection.
Interesting future work includes applying the inclusion detection method to other inverse problems. Similar stability estimates to that of Lemma \ref{lemma:stability-estimate-qpat} exist for the mathematically closely related problems of determining the absorption coefficient in acousto-optic imaging and the permittivity in microwave imaging, see \cite{choulli2021}. 
This is also the case for conductivity imaging in quantitative thermoacoustic tomography, where \cite{kocyigit2017} employed complex geometrical optics solutions. For the Calderón problem in two dimensions, \cite{clop2010} provides a stability estimate that is permitted for the star-shaped set parametrization, see also the comments after Theorem \ref{thm:star1} on the regularity of $\gamma$. There is a natural Hilbert space observation setting for the Calderón problem, see \cite{abraham2019}. Also in three dimensions and higher, conditional stability for inclusion detection in the context of the Calderón problem has been considered and shown to be logarithmic at best \cite{alessandrini2005}. The generalization to three dimensions and more complex phantoms is left for future work. An important direction in the numerical optimization of this approach is to consider gradient-based sampling methods. 


\section*{Acknowledgments}
The authors would like to thank Prof. Richard Nickl for many helpful discussions. The authors would also like to thank the Isaac Newton Institute for Mathematical Sciences, Cambridge,
for support and hospitality during the program Rich and Nonlinear Tomography where work on this paper was undertaken, supported by EPSRC grant no EP/R014604/1. BMA, KK and AKR were supported by The Villum Foundation (grant no. 25893). TT was supported by the European Research Council (ERC) under the European Union’s Horizon 2020 research and innovation programme (grant agreement no. 101001417 - QUANTOM) and the Academy of Finland (Centre of Excellence in Inverse Modelling and Imaging project 353086 and the Flagship Program Photonics Research and Innovation grant 320166).


\begin{appendices}
   \section{Covering numbers}\label{app:cov-numbers}
Consider a compact subset $A$ of a space $X$ endowed with a semimetric $d$. The covering number $N(A,d,\rho)$ denotes the minimum number of closed $d$-balls $\{x\in X: d(x_0,x)\leq \rho\}$ with center $x_0\in A$ and radius $\rho>0$ needed to cover $A$, see for example \cite[Appendix C]{ghosal2017} or \cite[Section 4.3.7]{gine2016}. Then the metric entropy is $\log N(A,d,\rho)$. When $d$ is replaced by a norm, we mean the metric induced by the norm.
\begin{lemma}\label{lemma:covering-numbers}
Let $(X,d_X)$ and $(Y,d_Y)$ be two linear spaces endowed with semimetric $d_X$ and $d_Y$.
    \begin{enumerate}[leftmargin = *, label=(\roman*)]
      \item If $f:X\rightarrow Y$ satisfies
      \begin{eqnarray}\nonumber d_Y(f(x),f(x')) \leq Cd_X(x,x')^\eta, \quad \forall x, x'\in A\end{eqnarray}
      for some $A\subset\subset X$ and some $\eta>0$, then for any $\rho > 0$ we have
        \begin{equation}
            N(f(A),d_Y,C\rho^\eta)\leq N(A,d_X,\rho).
        \end{equation}
      \item For $A\subset\subset X$ and $B\subset\subset Y$,
      \begin{eqnarray}\nonumber N(A\times B, d_\infty, \rho)\leq N(A, d_X, \rho)N(B,d_Y,\rho),\end{eqnarray}
    \end{enumerate}
    where $d_\infty((x,y),(x',y'))=\max(d_X(x,x'),d_Y(y,y'))$ is the product metric. 
\end{lemma}
\begin{proof}
    $(i)$ We denote by $B_X(x',\rho)$ and $B_Y(y',\rho)$ the ball in $X$ with center $x'\in X$ and radius $\rho>0$ and the ball in $Y$ with center $y'\in Y$ and radius $\rho>0$, respectively. For any $\rho>0$,
\begin{eqnarray}\nonumber f(B_X(x',\rho))\subset B_Y(f(x'),C\rho^\eta).\end{eqnarray}
Then it follows that
\begin{equation}
    N(f(A),d_Y,C\rho^\eta)\leq N(A,d_X,\rho).
\end{equation}
$(ii)$ Let $C_A$ be a finite set in $A$ and $C_B$ be a finite set in $B$ such that
\begin{eqnarray}\nonumber A\subset \bigcup_{x\in C_A} B_X(x,\rho) \quad \textnormal{ and } \quad B\subset \bigcup_{y\in C_B} B_Y(x,\rho).\end{eqnarray}
Take $z=(x,y)\in A\times B$, then there exists $x_0\in C_A$ such that $x\in B_X(x_0,\rho)$ and $y_0\in C_B$ such that $y\in B_Y(y_0,\rho)$. Hence $z\in B_{X\times Y}((x_0,y_0),\rho):=\{z\in X\times Y: d_\infty(z,(x_0,y_0))\leq \rho\}$. It follows that
\begin{eqnarray}\nonumber A\times B \subset \bigcup_{z\in C_A\times C_B} B_{X\times Y}(z,\rho),\end{eqnarray}
and hence the wanted property follows.
\end{proof}
\begin{lemma}\label{lemma:covering-lipschitz}
    Let $\mathcal{X}$ be a bounded Lipschitz domain in $\mathbb{R}^{d'}$ or the $d'$-dimensional torus and $\beta>d'/2$, then
    \begin{eqnarray}\nonumber \log N(\mathcal{S}_\beta(M),\|\cdot\|_{\infty},\rho)\leq C\rho^{-d'/\beta},\end{eqnarray}
    where $C=C(\beta,M,d',\mathcal{X})$ and $\|f\|_\infty :=\sup_{x\in \overline{\mathcal{X}}} |f(x)|$.
\end{lemma}
\begin{proof}
    Corollary 4.3.38 and the remark hereafter in \cite{gine2016} states that the norm ball $B_\beta(M)$ of the Sobolev space $H^\beta([0,1]^{d'})$ of radius $M$ satisfies for $\beta>d'/2$,
    \begin{equation}\label{eq:gine-and-nickl-entropy-bound}
        \log N(B_\beta(M),\|\cdot\|_{L^\infty([0,1]^{d'})},\rho)\leq C(\beta,M,d')\rho^{-d'/\beta}.
    \end{equation}
    If $\mathcal{X}$ is the $d'$-dimensional torus, we identify $H^\beta(\mathcal{X})$ with the corresponding periodic Sobolev space, which is a subset of $H^\beta([0,1]^{d'})$, hence the wanted result follows. 
    Now, if $\mathcal{X}$ is a bounded Lipschitz domain in $\mathbb{R}^{d'}$, we assume without loss of generality that $\mathcal{X}\subset [0,1]^{d'}$. Indeed, if $\mathcal{X}$ is not a subset of $[0,1]^{d'}$, we identify $f\in H^\delta(\mathcal{X})$ with $\tilde{f}\in H^\delta(\tilde{\mathcal{X}})$ for some $\tilde{\mathcal{X}}\subset [0,1]^{d'}$ by a scaling and update $M$ accordingly. Since $\mathcal{X}$ is Lipschitz, we let $E:H^\beta(\mathcal{X})\rightarrow H^\beta([0,1]^{d'})$ be a continuous extension operator satisfying
    \begin{eqnarray}\nonumber \|Ef\|_{H^\beta([0,1]^{d'})}\leq C(d',\beta,\mathcal{X})\|f\|_{H^\beta(\mathcal{X})},\end{eqnarray}
    see for example \cite{devore1993}. We denote the restriction $R:H^\beta([0,1]^{d'}))\rightarrow H^\beta(\mathcal{X})$, which is a contraction in supremum norm and the left-inverse of $E$. Then $\mathcal{S}_\beta(M)=R(E(\mathcal{S}_\beta(M)))$ and $E(\mathcal{S}_\beta(M))\subset B_\beta(CM)$, and hence
   \begin{eqnarray}\nonumber 
N(\mathcal{S}_\beta(M),\|\cdot\|_{\infty},\rho) &=  N(R(E(\mathcal{S}_\beta(M))),\|\cdot\|_{\infty},\rho) \\\nonumber 
    &\leq N(E(\mathcal{S}_\beta(M)),\|\cdot\|_{L^\infty([0,1]^{d'})},\rho),\\\nonumber 
       &\leq N(B_\beta(CM),\|\cdot\|_{L^\infty([0,1]^{d'})},\rho),\\ 
       &\leq C(\beta,M,d',\mathcal{X}) \rho^{-d'/\beta},\label{eq:cov-number-ext}
   \end{eqnarray}
   using also Lemma \ref{lemma:covering-numbers} $(i)$ and \eref{eq:gine-and-nickl-entropy-bound}.
\end{proof}
\begin{lemma}\label{lemma:triangle-inequality-argument}
With the notation defined in Section \ref{section:3}, we have for all $\rho>0$
    \begin{eqnarray}\nonumber N(B_\infty(M\rho)+ \mathcal{S}_\delta(CM), \|\cdot\|_{\infty}, 2M\rho)\leq N(\mathcal{S}_\delta(CM), \|\cdot\|_{\infty}, M\rho).\end{eqnarray}
\end{lemma}
\begin{proof}
By Lemma \ref{lemma:covering-lipschitz} there exists $N>0$ for which there is a sequence $\{\theta_i\}_{i=1}^N$ in $\mathcal{S}_\delta(CM)$ such that
\begin{eqnarray}\nonumber \mathcal{S}_\delta(CM) \subset \cup_{i=1}^N B_\infty(\theta_i,M\rho).\end{eqnarray} 
By the triangle inequality,
\begin{equation}\label{eq:triangle-sets}
    B_\infty(M\rho)+ \mathcal{S}_\delta(CM) \subset \cup_{i=1}^N B_\infty(\theta_i,2M\rho),
\end{equation}
since if $\theta = \theta^{(1)} + \theta^{(2)}$ for $\theta^{(1)}\in B_\infty(M\rho)$ and $\theta^{(2)}\in \mathcal{S}_\delta(CM)$, then there exists a $\theta_i$ such that $\|\theta^{(2)}-\theta_i\|_{\infty}\leq M\rho$, and hence
\begin{eqnarray}\nonumber \|\theta - \theta_i\|_{\infty} \leq \|\theta^{(1)}\|_{\infty} +  \|\theta^{(2)}-\theta_i\|_{\infty} \leq 2M\rho.\end{eqnarray}
Then the property follows from \eref{eq:triangle-sets}.
\end{proof}

    \section{On maximum likelihood composite testing}\label{app:testing}
     We denote by $E_n$ and $E_n^\gamma$ the expectation with respect to $P_n$ and $P_n^\gamma$ respectively. 
    \begin{lemma}\label{lemma:tests}
    Suppose for a non-increasing function $N(\rho)$, some $\rho_0>0$ and all $\rho > \rho_0$, we have
    \begin{eqnarray}\nonumber N(\{\gamma \in A_n: \rho < d_{\mathcal{G}}(\gamma,\gamma_0)< 2\rho\},d_{\mathcal{G}},\rho/4)\leq N(\rho).\end{eqnarray}
    Then for every $\rho > \rho_0$, there exist measurable functions $\Psi_n: \mathcal{Y}_{-}\rightarrow \{0,1\}$ such that
    \begin{eqnarray}\nonumber E_{n}^{\gamma_0}(\Psi_n)\leq N(\rho)\frac{e^{-\frac{1}{8\sigma^2}n\rho^2}}{1-e^{-\frac{1}{8\sigma^2}n\rho^2}},\end{eqnarray}
    \begin{eqnarray}\nonumber \sup_{\gamma \in A_n: d_{\mathcal{G}}(\gamma,\gamma_0)>\rho} E_{n}^{\gamma}(1-\Psi_n)\leq e^{-\frac{1}{32\sigma^2}n\rho^2}.\end{eqnarray}
\end{lemma}
\begin{proof}
    To construct the measurable functions $\Psi_n$ we use the maximum likelihood test, see \cite[Lemma D.16]{ghosal2017} and
    the covering argument of \cite[Theorem 7.1]{ghoshal2000}, see also \cite[Theorem 7.1.4]{gine2016}. Choose a finite set $S_j$ of points in each shell
    \begin{eqnarray}\nonumber S_j = \{\gamma \in A_n: \rho j < d_{\mathcal{G}}(\gamma,\gamma_0)\leq \rho(1+j)\}, \quad j\in \mathbb{N},\end{eqnarray}
    so that every $\gamma \in S_j$ is within distance $\frac{j\rho}{4}$ of a point in $S_j'$. For $\rho > \rho_0$, there are at most $N(j\rho)$ such points. For each $\gamma_{jl}\in S_j'$, define the measurable function,  $\Psi_{n,j,l}:\mathcal{Y}_{-}\rightarrow \{0,1\}$, known as the maximum likelihood test,
    \begin{eqnarray}\nonumber \Psi_{n,j,l}(y) := \mathds{1}_{A_{n,j,l}}(y),\end{eqnarray}
    where
    \begin{eqnarray}\nonumber A_{n,j,l} := \{y\in \mathcal{Y}_{-}: \frac{p_n^{\gamma_{jl}}}{p_n^{\gamma_{0}}}(y)>1 \}.\end{eqnarray}
    By \cite[Lemma D.16]{ghosal2017} we have
    \begin{eqnarray}\nonumber E_{n}^{\gamma_0}(\Psi_{n,j,l}) \leq e^{-\frac{1}{8\sigma^2}n(\rho j)^2}\end{eqnarray}
    and
    \begin{eqnarray}\nonumber \sup_{\{\gamma \in A_n:d_{\mathcal{G}}(\gamma,\gamma_{jl})\leq \frac{\rho j}{4}\}}E_{n}^{\gamma}(1-\Psi_{n,j,l}) \leq e^{-\frac{1}{32\sigma^2}n(\rho j)^2}.\end{eqnarray}
    Now, set $\Psi_n(y):=\mathds{1}_{\cup_{j,l}A_{n,j,l}}(y)$. This is also a measurable function, since a countable union of measurable sets is measurable. Then by the union bound
    \begin{eqnarray}\nonumber  \fl
        E_{n}^{\gamma_0}(\Psi_n) &\leq \sum_{j\in\mathbb{N}} \sum_{l=1}^{N(j\rho)} E_{n}^{\gamma_0}(\Psi_{n,j,l}) \leq \sum_{j\in\mathbb{N}} N(j\rho) e^{-\frac{1}{8\sigma^2}n(\rho j)^2}
        \leq N(\rho)\frac{e^{-\frac{1}{8\sigma^2}n\rho^2}}{1-e^{-\frac{1}{8\sigma^2}n\rho^2}}
    \end{eqnarray}
    On the other hand, for any $j\geq 1$,
    \begin{eqnarray}\nonumber 
        \sup_{\gamma\in \cup_{i\geq j}S_i} E_{n}^{\gamma_0}(1-\Psi_n) &= \sup_{i\geq j,l} \sup_{\gamma: d_{\mathcal{G}}(\gamma,\gamma_{i,l}) \leq \frac{\rho i}{4}}P_n^{\gamma_0}(\cap_{j',l'} A_{n,j',l'}^c),\\\nonumber 
        &\leq \sup_{i\geq j,l} \sup_{\gamma: d_{\mathcal{G}}(\gamma,\gamma_{i,l}) \leq \frac{\rho i}{4}}E_{n}^{\gamma_0}(1-\Psi_{n,i,l}),\\\nonumber 
        &\leq \sup_{i\geq j} e^{-\frac{1}{32\sigma^2}n(\rho i)^2}.
    \end{eqnarray}
    For $j=1$ we get the wanted result.
\end{proof}
There are other ways to prove the existence of suitable measurable functions $\Psi_n$ used in the proof of Theorem \ref{thm:contr}. We mention here the approximation argument of  \cite[Lemma 8]{abraham2019} that requires smoothness properties of $\mathcal{G}$. 
\begin{proof}[Proof of \ref{thm:contr}]
    For the convenience of the reader, we provide what is a standard testing argument for our setting in Lemma \ref{lemma:tests}, see also \cite[Theorem 7.1]{ghoshal2000}, implied by Condition \ref{cond:1.3}. Indeed, since the covering number decreases, when increasing the `radius', we have for all $\rho>\rho_0:= 4m_0r_n$,
\begin{eqnarray}\nonumber N(A_n,d_{\mathcal{G}},\frac{\rho}{4})\leq N(\rho):=e^{C_3nr_n^2}.\end{eqnarray} 
Given any $C_2 > C_1 + 4$, we set $\rho := 4mr_n$ for $m>m_0$ large enough (depending also on $C_2$, $C_3$ and $\sigma^2$ by the following) and apply Lemma \ref{lemma:tests}: there exists measurable functions $\Psi_n: \mathcal{Y}_{-}\rightarrow \{0,1\}$ such that 
\begin{eqnarray}\nonumber 
    E_n^{\gamma_0}(\Psi_n)
    &\leq e^{C_3nr_n^2}\frac{e^{-2\sigma^{-2}mn r_n^2}}{1-e^{-2\sigma^{-2}m n r_n^2}} \leq e^{-C_2nr_n^2}.
\end{eqnarray}
In addition, choosing $m$ such that $(4m)^2/(32\sigma^2)\geq C_2$ we note
\begin{eqnarray}\nonumber 
    \sup_{\gamma \in A_n: d_{\mathcal{G}}(\gamma,\gamma_0)>4mr_n^2} E_{n}^{\gamma}(1-\Psi_n)&\leq e^{-C_2nr_n^2}.
\end{eqnarray}
Then Theorem 28 in \cite{nickl2020} and modifications as in the proof of Theorem 1.3.2 in \cite{nickl2022} give the claim. 
\end{proof}

\section{Proofs of section \ref{section:4}}\label{sec:extra-proofs}
\begin{proof}[Proof of Theorem \ref{thm:star}]
    The proof relies on satisfying Condition \ref{cond:1.1}, \ref{cond:1.2} and \ref{cond:1.3} for the choice $A_n= \Phi(\Theta_n)$ for 
\begin{equation}\label{eq:reg-sets-product}
    \Theta_n := \{\theta = \phi_1+\phi_2: \|\phi_1\|_{\infty}\leq M\bar{r}_n, \|\phi_2\|_{\mathcal{H}}\leq M\}^{\mathcal{N}} \cap \mathcal{S}_\beta(M),
\end{equation}
where $\mathcal{S}_\beta(M)$ is defined in \eref{eq:Theta-S-def}. To satisfy Condition \ref{cond:1.1} we follow Lemma \ref{lemma:small-ball} and note for $\theta=(\theta_1,\ldots,\theta_{\mathcal{N}})$ and $\theta_0=(\theta_{0,1},\ldots,\theta_{0,\mathcal{N}})$ that 
\begin{eqnarray}\nonumber \fl
    \{\theta\in \Theta: d_{\mathcal{G}}(\Phi(\theta),\Phi(\theta_0))\leq r_n\}\\\nonumber 
    \supset \{\theta\in \Theta:\|\theta-\theta_0\|_\mathcal{N} \leq C\bar{r}_n, \|\theta_i-\theta_{0,i}\|_{H^\beta(\mathbb{T})}\leq \tilde{R}, i=1,\ldots,\mathcal{N}\},\\\nonumber 
    \supset \{\theta\in \Theta:\|\theta_i-\theta_{0,i}\|_\mathcal{N} \leq C\bar{r}_n, \|\theta_i-\theta_{0,i}\|_{H^\beta(\mathbb{T})}\leq \tilde{R}, i=1,\ldots,\mathcal{N}\},\\\nonumber 
    \supset \otimes_{i=1}^\mathcal{N}\left(\{\theta_i:\|\theta_i-\theta_{0,i}\|_{L^\infty(\mathbb{T})}\} \cap \{\theta_i:\|\theta_i-\theta_{0,i}\|_{H^\beta(\mathbb{T})}\leq \tilde{R}\} \right), 
\end{eqnarray}
for some $\tilde{R}>0$ chosen sufficiently large. Then \eref{eq:product-prior} together with the argument in Lemma \ref{lemma:small-ball} implies that Condition \ref{cond:1.1} is satisfied. Note $\Theta_n$ in \eref{eq:reg-sets-product} is the $\mathcal{N}$-product set of \eref{eq:reg-sets}. Then repeated use of the standard set relation $A^2\setminus B^2=[A \times (A\setminus B)]\cup[(A\setminus B)\times A]$ and the argument of Lemma \ref{lemma:excess-mass-2} implies there exists $M>C(C_2,\Pi'_\theta,\delta,\mathcal{N})$ such that 
    \begin{eqnarray}\nonumber \Pi_\theta(\Theta\setminus\Theta_n)\leq e^{-C_2 nr_n^2},\end{eqnarray}
    for any given $C_2>0$, hence Condition \ref{cond:1.2} is satisfied. Condition \ref{cond:1.3} is satisfied as in Lemma \ref{lemma:metric-entropy-reg-sets} using Lemma \ref{lemma:covering-numbers} $(ii)$. Then the result follows as in the proof of Theorem \ref{thm:consist-in-Hbeta} and Corollary \ref{corollary:post-mean}.
\end{proof}
\begin{lemma}\label{lemma:level-set-app}
    Let $V_\epsilon$ be defined as in \eref{eq:Veps} for $\theta_0\in H^\beta_\diamond(\mathcal{X})$, $\beta>1+d'/2$ and some $c=c_{i-1}\in \overline{\mathbb{R}}$. Then for $\epsilon>0$ sufficiently small
    \begin{eqnarray}\nonumber |V_\epsilon|\leq C(\theta_0,c,\mathcal{X})\epsilon.\end{eqnarray}
\end{lemma}
\begin{proof}
    Note for $c=c_0=-\infty$ and $c=c_\mathcal{N}=\infty$ this is trivially satisfied. 
    Next, the inverse function theorem implies that any point $x_0\in \theta^{-1}(c)$ has a neighborhood $N_{x_0}$ that is a diffeomorphic image $\varphi_{x_0}(Q_{\epsilon_{x_0}})$ of a box
   \begin{eqnarray}\nonumber Q_{\epsilon_{x_0}} := \{(s,t):|s|\leq \epsilon_{x_0}, |t|\leq \epsilon_{x_0}\}, \quad 0<\epsilon_{x_0}<1,\end{eqnarray}
    such that
    \begin{eqnarray}\nonumber \theta_0(\varphi_{x_0}(s,t)) = s|(\nabla\theta_0)(x_0)|+c,\end{eqnarray}
    (one should find the inverse of $g(x_1,x_2)=(\frac{\theta_0(x_1,x_2)-c}{|\nabla\theta_0(x_0)|},x_2)$ in a neighborhood of $x_0$). Note we have a $C^1$ parametrization of an intersection of $V_\epsilon$ with a small neighborhood of $x_0$,
    \begin{eqnarray}\nonumber V_{\epsilon}\cap N_{x_0}=\{\varphi_{x_0}(s,t): |s|\leq \frac{\epsilon}{|\nabla\theta_0(x_0)|},|t|\leq \epsilon_{x_0}\},\end{eqnarray}
    for all $\epsilon \leq |\nabla \theta_0(x_0)|\epsilon_{x_0}$. By the classical area formula, we have
    \begin{eqnarray}\nonumber |V_{\epsilon}\cap N_{x_0}| = \int_{|s|\leq C(x_0,\theta_0) \epsilon} \int_{|t|\leq \epsilon_{x_0}} |J\varphi_{x_0}(s,t)| \, ds \, dt \leq C(x_0,\theta_0) \epsilon,\end{eqnarray}
    since the continuous function $J\varphi_{x_0}$ (it is a polynomial of zero'th and first order derivatives of $\varphi_{x_0}$), is integrated on a compact domain. Note $\cup_{x_0\in \mathcal{I}}N_{x_0}$ is an open cover of $\theta_0^{-1}(c)$ for some finite set $\mathcal{I}\subset \theta_0^{-1}(c)$ depending on $\mathcal{X}$ and $\theta_0$. Take $\epsilon$ such that $V_\epsilon \subset \cup_{x_0\in\mathcal{I}}N_{x_0}$. This $\epsilon$ exists since $\theta_0$ as defined on $\overline{\mathcal{X}}$ is a closed function, and hence there exists in $\mathbb{R}$ a neighborhood $U$ of $c$ such that $\theta_0^{-1}(U)\subset \cup_{x_0\in \mathcal{I}} N_{x_0}$, see \cite[Theorem 1.4.13]{engelking1989}. Then,
    \begin{eqnarray}\nonumber 
        |V_\epsilon| &\leq \sum_{x_0\in \mathcal{I}} |V_\epsilon \cap N_{x_0}|\leq C(\theta_0,c,\mathcal{X})\epsilon.
    \end{eqnarray}
    This is true for any $i=1,\ldots,\mathcal{N}-1$ for which the estimate is only updated by a new constant.
\end{proof}

\begin{proof}[Proof of Theorem \ref{thm:level-set-const}]
    Let $\gamma_0^n = \Phi_{n^{-k}}(\theta_0)$ and $\gamma^n = \Phi_{n^{-k}}(\theta)$ for some $0<k<1$, which we will choose later. For any $\hat{r}_n>0$, the triangle inequality gives
    \begin{eqnarray}\nonumber \fl \left\{\gamma: \|\gamma-\gamma_0\|_{L^2(D)}\leq C_0\hat{r}_n\right\}\\ \nonumber
    \supset \left\{\gamma: \|\gamma-\gamma_0^n\|_{L^2(D)}\leq \frac{1}{2}C_0\hat{r}_n, \|\gamma_0^n-\gamma_0\|_{L^2(D)}\leq \frac{1}{2}C_0\hat{r}_n\right\},\end{eqnarray}
    hence
    \begin{eqnarray} \fl \nonumber 
        \Pi(\gamma:\|\gamma-\gamma_0\|_{L^2(D)}\leq C_0\hat{r}_n |Y) &\geq \Pi(\gamma:\|\gamma-\gamma_0^n\|_{L^2(D)}\leq \frac{1}{2}C_0\hat{r}_n|Y)\\
        &\,\,\,\,\,\times \mathds{1}_{\|\gamma_0^n-\gamma_0\|_{L^2(D)}\leq \frac{1}{2}C_0\hat{r}_n}. \label{eq:triangle-argument}
    \end{eqnarray}
    We shall consider the two factors of the right-hand side in separate parts below:

    \noindent 1) We check that Condition \ref{cond:1.2}, \ref{cond:1.3}, and \ref{cond:1.1} are satisfied for the choice $A_n=\Phi(\Theta_n)$ with
\begin{equation}\label{eq:reg-sets-level} 
    \Theta_n := \{\theta = \theta_1+\theta_2: \|\theta_1\|_{\infty}\leq Mr_n^{\frac{1}{\eta}}n^{-k}, \|\theta_2\|_{\mathcal{H}}\leq M\} \cap \mathcal{S}_\beta(M),
\end{equation}
and $k$ to be chosen below. For \ref{cond:1.2} it is clear from \eref{eq:level-set-critical-points} that for each $n$,
    \begin{eqnarray}\nonumber \Pi_\theta(\Theta)=1.\end{eqnarray}
    As in the proof of Lemma \ref{lemma:excess-mass-2} there exists $M>C(C_2,\Pi_\theta',\delta)$ such that Condition \ref{cond:1.2} is satisfied, if $a=a(k)$ is such that
    \begin{equation}\label{eq:eq-on-a-level}
        (r_n^{\frac{1}{\eta}}n^{-k}n^{1/2-a})^{-b}= nr_n^2,
    \end{equation}
    for $b=\frac{2d}{2\delta-d}$. This is satisfied when
    \begin{equation}\label{eq:cond-on-a}
            a = a(k) = \frac{\eta(\delta-dk)}{2\delta\eta+d}, \quad \textnormal{and} \quad 0<k<\frac{\delta}{d},
    \end{equation}
    so that $0<a<1/2$. Condition \ref{cond:1.3} follows as in the proof of Lemma \ref{lemma:covering-numbers} with $\bar{r}_n$ replaced with $r_n^{\frac{1}{\eta}}n^{-k}$. Again it reduces to the covering number of the norm-ball in $H^\delta(\mathcal{X})$ for which we need $a$ such that
    \begin{eqnarray}\nonumber (r_n^{\frac{1}{\eta}}n^{-k})^{-d/\delta} = nr_n^2\end{eqnarray}
    as in \eref{eq:condition-on-a-covering}. This is indeed satisfied by \eref{eq:cond-on-a}. For Condition \ref{cond:1.1} we proceed as in the proof of Lemma \ref{lemma:small-ball} and use Lemma \ref{lemma:4.5} to obtain
    \begin{eqnarray}\nonumber \fl
		\{\theta \in \Theta: d_{\mathcal{G}}(\gamma^n,&\gamma_0^n)\leq r_n\}\supset \{\theta\in \Theta: \|\theta-\theta_0\|_{\infty} \leq Cn^{-k} r_{n}^{\frac{1}{\eta}}\} \cap \mathcal{S}_\beta(R), 
	\end{eqnarray}
    where $C=C(\eta,C_{\mathcal{G}},C_{\Phi},R)$. Continuing the argument and using \eref{eq:level-set-critical-points}, Condition \ref{cond:1.1} is satisfied for some $C_1>0$ if again $a$ satisfies \eref{eq:eq-on-a-level}. By Theorem \ref{thm:contr}
    \begin{eqnarray}\nonumber \Pi(B_{\mathcal{G}}(\gamma_0^n,Cr_n) \cap A_n | Y) \rightarrow 1 \quad \textnormal{ in $P_n^{\gamma_0}$-probability},\end{eqnarray}
    as $n \rightarrow \infty$ for some constant $C>0$. It follows that
\begin{eqnarray}\nonumber \Pi(\gamma: \|\gamma-\gamma_0^n\|_{L^2(D)} \leq Cr_n^\nu | Y) \rightarrow 1 \quad \textnormal{ in $P_n^{\gamma_0}$-probability},
\end{eqnarray}
with rate $e^{-bnr_n^2}$, $0<b<C_2-C_1-4$ as $n\rightarrow \infty$ as in Theorem \ref{thm:consist-in-Hbeta}.\\
\noindent 2) For the second factor, note that $\theta_0\in H_\diamond^\beta(\mathcal{X})$ and Lemma \ref{lemma:4.5} (i) implies
\begin{eqnarray}\nonumber \|\gamma_0^n-\gamma_0\|_{L^2(D)}\leq C'(\theta_0,\mathcal{X},D,\mathbf{c})n^{-k/2}.\end{eqnarray}
Since $r_n^\nu=n^{-a(k)\nu}$ is a strictly increasing function of $k$ (the rate becomes worse for larger $k$) and $n^{-k/2}$ is strictly decreasing in $k$, the optimal choice of $k$ satisfies $r_n^\nu = n^{-k/2}$, which is solved by
    \begin{equation}\label{eq:cond-on-k}
        k = \frac{2\delta \eta \nu}{2d\eta \nu +2\delta \eta + d},
    \end{equation}
which also satisfies the condition on $k$ in \eref{eq:cond-on-a}, since $\delta>d$. Inserting this back into \eref{eq:cond-on-a} yields \eref{eq:new-a}. Finally, take $C_0 = 2\max(C,C')$ and $\hat{r}_n = r_n^\nu$ and note by \eref{eq:triangle-argument} that
\begin{eqnarray}\nonumber \fl \Pi(\gamma:\|\gamma-\gamma_0\|_{L^2(D)}\leq C_0 r_n^\nu |Y) \geq \Pi(\gamma:\|\gamma-\gamma_0^n\|_{L^2(D)}\leq \frac{1}{2}C_0r_n^\nu|Y)\rightarrow 1,\end{eqnarray}
in $P_n^{\gamma_0}$-probability as $n\rightarrow \infty$. Then the wanted result follows as in Corollary \ref{corollary:post-mean}.
\end{proof}

\end{appendices}

\bibliographystyle{unsrt} 
\bibliography{refs} 

\begin{thebibliography}{10}

\bibitem{kaipio2005}
Jari Kaipio and Erkki Somersalo.
\newblock {\em Statistical and computational inverse problems}, volume 160 of
  {\em Applied Mathematical Sciences}.
\newblock Springer-Verlag, New York, 2005.

\bibitem{stuart2010}
A.~M. Stuart.
\newblock Inverse problems: a {B}ayesian perspective.
\newblock {\em Acta Numer.}, 19:451--559, 2010.

\bibitem{cherepenin2001}
V~Cherepenin, A~Karpov, A~Korjenevsky, V~Kornienko, A~Mazaletskaya, D~Mazourov,
  and D~Meister.
\newblock A 3d electrical impedance tomography ({EIT}) system for breast cancer
  detection.
\newblock {\em Physiological Measurement}, 22(1):9--18, feb 2001.

\bibitem{xu2006}
Minghua Xu and Lihong~V. Wang.
\newblock {Photoacoustic imaging in biomedicine}.
\newblock {\em Review of Scientific Instruments}, 77(4):041101, 04 2006.

\bibitem{hallaji2014}
Milad Hallaji, Aku Seppänen, and Mohammad Pour-Ghaz.
\newblock Electrical impedance tomography-based sensing skin for quantitative
  imaging of damage in concrete.
\newblock {\em Smart Materials and Structures}, 23(8):085001, jun 2014.

\bibitem{fuch2021}
Patrick Fuchs, Thorben Kröger, and Christoph~S. Garbe.
\newblock Defect detection in ct scans of cast aluminum parts: A machine vision
  perspective.
\newblock {\em Neurocomputing}, 453:85--96, 2021.

\bibitem{bora2017}
Ashish Bora, Ajil Jalal, Eric Price, and Alexandros~G. Dimakis.
\newblock Compressed sensing using generative models.
\newblock In {\em Proceedings of the 34th International Conference on Machine
  Learning - Volume 70}, ICML'17, page 537–546. JMLR.org, 2017.

\bibitem{dashti2017}
Masoumeh Dashti and Andrew~M. Stuart.
\newblock The {B}ayesian approach to inverse problems.
\newblock In {\em Handbook of uncertainty quantification. {V}ol. 1, 2, 3},
  pages 311--428. Springer, Cham, 2017.

\bibitem{buithanh2014}
Tan Bui-Thanh and Omar Ghattas.
\newblock An analysis of infinite dimensional bayesian inverse shape acoustic
  scattering and its numerical approximation.
\newblock {\em SIAM/ASA Journal on Uncertainty Quantification}, 2(1):203--222,
  2014.

\bibitem{dunlop2016}
Matthew~M. Dunlop and Andrew~M. Stuart.
\newblock The {B}ayesian formulation of {EIT}: analysis and algorithms.
\newblock {\em Inverse Probl. Imaging}, 10(4):1007--1036, 2016.

\bibitem{igelsias2016}
Marco~A. Iglesias, Yulong Lu, and Andrew Stuart.
\newblock A {B}ayesian level set method for geometric inverse problems.
\newblock {\em Interfaces Free Bound.}, 18(2):181--217, 2016.

\bibitem{afkham2023}
Babak~Maboudi Afkham, Yiqiu Dong, and Per~Christian Hansen.
\newblock Uncertainty quantification of inclusion boundaries in the context of
  {X}-ray tomography.
\newblock {\em SIAM/ASA J. Uncertain. Quantif.}, 11(1):31--61, 2023.

\bibitem{carpio2020}
Ana Carpio, Sergei Iakunin, and Georg Stadler.
\newblock Bayesian approach to inverse scattering with topological priors.
\newblock {\em Inverse Problems}, 36(10):105001, 29, 2020.

\bibitem{borggaard2023}
Jeff Borggaard, Nathan~E Glatt-Holtz, and Justin Krometis.
\newblock A statistical framework for domain shape estimation in stokes flows.
\newblock {\em Inverse Problems}, 39(8):085009, jun 2023.

\bibitem{yin2022}
Yunwen Yin, Weishi Yin, Pinchao Meng, and Hongyu Liu.
\newblock The interior inverse scattering problem for a two-layered cavity
  using the {B}ayesian method.
\newblock {\em Inverse Probl. Imaging}, 16(4):673--690, 2022.

\bibitem{yan2020}
Zhipeng Yang, Xinping Gui, Ju~Ming, and Guanghui Hu.
\newblock Bayesian approach to inverse time-harmonic acoustic scattering with
  phaseless far-field data.
\newblock {\em Inverse Problems}, 36(6):065012, 30, 2020.

\bibitem{dunlop2017}
Matthew~M. Dunlop, Marco~A. Iglesias, and Andrew~M. Stuart.
\newblock Hierarchical {B}ayesian level set inversion.
\newblock {\em Stat. Comput.}, 27(6):1555--1584, 2017.

\bibitem{chada2018}
Neil~K. Chada, Marco~A. Iglesias, Lassi Roininen, and Andrew~M. Stuart.
\newblock Parameterizations for ensemble {K}alman inversion.
\newblock {\em Inverse Problems}, 34(5):055009, 31, 2018.

\bibitem{huang2021a}
Jiangfeng Huang, Zhaoxing Li, and Bo~Wang.
\newblock A {B}ayesian level set method for the shape reconstruction of inverse
  scattering problems in elasticity.
\newblock {\em Comput. Math. Appl.}, 97:18--27, 2021.

\bibitem{huang2021b}
Jiangfeng Huang, Zhiliang Deng, and Liwei Xu.
\newblock A {B}ayesian level set method for an inverse medium scattering
  problem in acoustics.
\newblock {\em Inverse Probl. Imaging}, 15(5):1077--1097, 2021.

\bibitem{reese2021}
William Reese, Arvind~K Saibaba, and Jonghyun Lee.
\newblock Bayesian level set approach for inverse problems with piecewise
  constant reconstructions.
\newblock {\em arXiv preprint arXiv:2111.15620}, 2021.

\bibitem{monard2021}
Fran\c{c}ois Monard, Richard Nickl, and Gabriel~P. Paternain.
\newblock Consistent inversion of noisy non-{A}belian {X}-ray transforms.
\newblock {\em Comm. Pure Appl. Math.}, 74(5):1045--1099, 2021.

\bibitem{ghoshal2000}
Subhashis Ghosal, Jayanta~K. Ghosh, and Aad~W. van~der Vaart.
\newblock {Convergence rates of posterior distributions}.
\newblock {\em The Annals of Statistics}, 28(2):500 -- 531, 2000.

\bibitem{nickl2022}
Richard Nickl.
\newblock {\em Bayesian non-linear statistical inverse problems}.
\newblock Zurich Lectures in Advanced Mathematics. EMS Press, Berlin, 2023.

\bibitem{agapiou2013}
Sergios Agapiou, Stig Larsson, and Andrew~M. Stuart.
\newblock Posterior contraction rates for the {B}ayesian approach to linear
  ill-posed inverse problems.
\newblock {\em Stochastic Process. Appl.}, 123(10):3828--3860, 2013.

\bibitem{choulli2021}
Mourad Choulli.
\newblock Some stability inequalities for hybrid inverse problems.
\newblock {\em C. R. Math. Acad. Sci. Paris}, 359:1251--1265, 2021.

\bibitem{bal2011}
Guillaume Bal, Kui Ren, Gunther Uhlmann, and Ting Zhou.
\newblock Quantitative thermo-acoustics and related problems.
\newblock {\em Inverse Problems}, 27(5):055007, 15, 2011.

\bibitem{abraham2019}
Kweku Abraham and Richard Nickl.
\newblock On statistical {C}alder\'{o}n problems.
\newblock {\em Math. Stat. Learn.}, 2(2):165--216, 2019.

\bibitem{giordano2020}
Matteo Giordano and Richard Nickl.
\newblock Consistency of {B}ayesian inference with {G}aussian process priors in
  an elliptic inverse problem.
\newblock {\em Inverse Problems}, 36(8):085001, 35, 2020.

\bibitem{nickl2020b}
Richard Nickl, Sara van~de Geer, and Sven Wang.
\newblock Convergence rates for penalized least squares estimators in {PDE}
  constrained regression problems.
\newblock {\em SIAM/ASA J. Uncertain. Quantif.}, 8(1):374--413, 2020.

\bibitem{nickl2020}
Richard Nickl.
\newblock Bernstein--von {M}ises theorems for statistical inverse problems {I}:
  {S}chr\"{o}dinger equation.
\newblock {\em J. Eur. Math. Soc. (JEMS)}, 22(8):2697--2750, 2020.

\bibitem{gine2016}
Evarist Gin\'{e} and Richard Nickl.
\newblock {\em Mathematical foundations of infinite-dimensional statistical
  models}.
\newblock Cambridge Series in Statistical and Probabilistic Mathematics, [40].
  Cambridge University Press, New York, 2016.

\bibitem{ghosal2017}
Subhashis Ghosal and Aad van~der Vaart.
\newblock {\em Fundamentals of nonparametric {B}ayesian inference}, volume~44
  of {\em Cambridge Series in Statistical and Probabilistic Mathematics}.
\newblock Cambridge University Press, Cambridge, 2017.

\bibitem{diestel1977}
J.~Diestel and J.~J. Uhl, Jr.
\newblock {\em Vector measures}.
\newblock Mathematical Surveys, No. 15. American Mathematical Society,
  Providence, R.I., 1977.
\newblock With a foreword by B. J. Pettis.

\bibitem{dudley1989}
Richard~M. Dudley.
\newblock {\em Real analysis and probability}.
\newblock The Wadsworth \& Brooks/Cole Mathematics Series. Wadsworth \&
  Brooks/Cole Advanced Books \& Software, Pacific Grove, CA, 1989.

\bibitem{vollmer2013}
Sebastian~J. Vollmer.
\newblock Posterior consistency for {B}ayesian inverse problems through
  stability and regression results.
\newblock {\em Inverse Problems}, 29(12):125011, 32, 2013.

\bibitem{cotter2013}
S.~L. Cotter, G.~O. Roberts, A.~M. Stuart, and D.~White.
\newblock M{CMC} methods for functions: modifying old algorithms to make them
  faster.
\newblock {\em Statist. Sci.}, 28(3):424--446, 2013.

\bibitem{hairer2014}
Martin Hairer, Andrew~M. Stuart, and Sebastian~J. Vollmer.
\newblock Spectral gaps for a {M}etropolis-{H}astings algorithm in infinite
  dimensions.
\newblock {\em Ann. Appl. Probab.}, 24(6):2455--2490, 2014.

\bibitem{dunlop2020}
Matthew~M. Dunlop, Tapio Helin, and Andrew~M. Stuart.
\newblock Hyperparameter estimation in {B}ayesian {MAP} estimation:
  parameterizations and consistency.
\newblock {\em SMAI J. Comput. Math.}, 6:69--100, 2020.

\bibitem{roininen2014}
Lassi Roininen, Janne M.~J. Huttunen, and Sari Lasanen.
\newblock Whittle-{M}at\'{e}rn priors for {B}ayesian statistical inversion with
  applications in electrical impedance tomography.
\newblock {\em Inverse Probl. Imaging}, 8(2):561--586, 2014.

\bibitem{engl1996}
Heinz~W. Engl, Martin Hanke, and Andreas Neubauer.
\newblock {\em Regularization of inverse problems}, volume 375 of {\em
  Mathematics and its Applications}.
\newblock Kluwer Academic Publishers Group, Dordrecht, 1996.

\bibitem{hairer2009}
Martin Hairer.
\newblock An introduction to stochastic pdes.
\newblock {\em arXiv preprint arXiv:0907.4178}, 2009.

\bibitem{vaart2009}
A.~W. van~der Vaart and J.~H. van Zanten.
\newblock Adaptive {B}ayesian estimation using a {G}aussian random field with
  inverse gamma bandwidth.
\newblock {\em Ann. Statist.}, 37(5B):2655--2675, 2009.

\bibitem{li1999}
Wenbo~V. Li and Werner Linde.
\newblock Approximation, metric entropy and small ball estimates for {G}aussian
  measures.
\newblock {\em Ann. Probab.}, 27(3):1556--1578, 1999.

\bibitem{schymura2014}
Daria Schymura.
\newblock An upper bound on the volume of the symmetric difference of a body
  and a congruent copy.
\newblock {\em Adv. Geom.}, 14(2):287--298, 2014.

\bibitem{sickel1999}
Winfried Sickel.
\newblock Pointwise multipliers of {L}izorkin-{T}riebel spaces.
\newblock In {\em The {M}az'ya anniversary collection, {V}ol. 2 ({R}ostock,
  1998)}, volume 110 of {\em Oper. Theory Adv. Appl.}, pages 295--321.
  Birkh\"{a}user, Basel, 1999.

\bibitem{faraco2013}
Daniel Faraco and Keith~M. Rogers.
\newblock The {S}obolev norm of characteristic functions with applications to
  the {C}alder\'{o}n inverse problem.
\newblock {\em Q. J. Math.}, 64(1):133--147, 2013.

\bibitem{makai1959}
E~Makai.
\newblock Steiner type inequalities in plane geometry.
\newblock {\em Periodica Polytechnica Electrical Engineering (Archives)},
  3(4):345--355, 1959.

\bibitem{gray2004}
Alfred Gray.
\newblock {\em Tubes}, volume 221 of {\em Progress in Mathematics}.
\newblock Birkh\"{a}user Verlag, Basel, second edition, 2004.
\newblock With a preface by Vicente Miquel.

\bibitem{kallenberg2021}
Olav Kallenberg.
\newblock {\em Foundations of modern probability}, volume~99 of {\em
  Probability Theory and Stochastic Modelling}.
\newblock Springer, Cham, [2021] \copyright 2021.
\newblock Third edition [of 1464694].

\bibitem{walther1997}
Guenther Walther.
\newblock Granulometric smoothing.
\newblock {\em Ann. Statist.}, 25(6):2273--2299, 1997.

\bibitem{azais2009}
Jean-Marc Aza\"{\i}s and Mario Wschebor.
\newblock {\em Level sets and extrema of random processes and fields}.
\newblock John Wiley \& Sons, Inc., Hoboken, NJ, 2009.

\bibitem{rasmussen2006}
Carl~Edward Rasmussen and Christopher K.~I. Williams.
\newblock {\em Gaussian processes for machine learning}.
\newblock Adaptive Computation and Machine Learning. MIT Press, Cambridge, MA,
  2006.

\bibitem{tarvainen2012}
T.~Tarvainen, B.~T. Cox, J.~P. Kaipio, and S.~R. Arridge.
\newblock Reconstructing absorption and scattering distributions in
  quantitative photoacoustic tomography.
\newblock {\em Inverse Problems}, 28(8):084009, 17, 2012.

\bibitem{kuchment2012}
Peter Kuchment.
\newblock Mathematics of hybrid imaging: A brief review.
\newblock In Irene Sabadini and Daniele~C Struppa, editors, {\em The
  Mathematical Legacy of Leon Ehrenpreis}, pages 183--208, Milano, 2012.
  Springer Milan.

\bibitem{evans1998}
Lawrence~C. Evans.
\newblock {\em Partial differential equations}, volume~19 of {\em Graduate
  Studies in Mathematics}.
\newblock American Mathematical Society, Providence, RI, 1998.

\bibitem{gilbarg83}
David Gilbarg and Neil~S. Trudinger.
\newblock {\em Elliptic partial differential equations of second order}, volume
  224 of {\em Grundlehren der mathematischen Wissenschaften [Fundamental
  Principles of Mathematical Sciences]}.
\newblock Springer-Verlag, Berlin, second edition, 1983.

\bibitem{lions1972}
J.-L. Lions and E.~Magenes.
\newblock {\em Non-homogeneous boundary value problems and applications. {V}ol.
  {I}}.
\newblock Die Grundlehren der mathematischen Wissenschaften, Band 181.
  Springer-Verlag, New York-Heidelberg, 1972.
\newblock Translated from the French by P. Kenneth.

\bibitem{kunyansky2007}
Leonid~A. Kunyansky.
\newblock A series solution and a fast algorithm for the inversion of the
  spherical mean {R}adon transform.
\newblock {\em Inverse Problems}, 23(6):S11--S20, 2007.

\bibitem{agranovsky2007}
Mark Agranovsky and Peter Kuchment.
\newblock Uniqueness of reconstruction and an inversion procedure for
  thermoacoustic and photoacoustic tomography with variable sound speed.
\newblock {\em Inverse Problems}, 23(5):2089--2102, 2007.

\bibitem{hänninen2018}
Niko Hänninen, Aki Pulkkinen, and Tanja Tarvainen.
\newblock Image reconstruction with reliability assessment in quantitative
  photoacoustic tomography.
\newblock {\em Journal of Imaging}, 4(12), 2018.

\bibitem{lindgren2011}
Finn Lindgren, H\aa~vard Rue, and Johan Lindstr\"{o}m.
\newblock An explicit link between {G}aussian fields and {G}aussian {M}arkov
  random fields: the stochastic partial differential equation approach.
\newblock {\em J. R. Stat. Soc. Ser. B Stat. Methodol.}, 73(4):423--498, 2011.
\newblock With discussion and a reply by the authors.

\bibitem{taylor2011}
Michael~E. Taylor.
\newblock {\em Partial differential equations {I}. {B}asic theory}, volume 115
  of {\em Applied Mathematical Sciences}.
\newblock Springer, New York, second edition, 2011.

\bibitem{kepner2020}
Jeremy Kepner, Andreas Kipf, Darren Engwirda, Navin Vembar, Michael Jones,
  Lauren Milechin, Vijay Gadepally, Chris Hill, Tim Kraska, William Arcand,
  David Bestor, William Bergeron, Chansup Byun, Matthew Hubbell, Michael Houle,
  Andrew Kirby, Anna Klein, Julie Mullen, Andrew Prout, and Peter Michaleas.
\newblock Fast mapping onto census blocks.
\newblock In {\em 2020 IEEE High Performance Extreme Computing Conference
  (HPEC)}, pages 1--8. IEEE, 2020.

\bibitem{khristenko2019}
U.~Khristenko, L.~Scarabosio, P.~Swierczynski, E.~Ullmann, and B.~Wohlmuth.
\newblock Analysis of boundary effects on {PDE}-based sampling of
  {W}hittle-{M}at\'{e}rn random fields.
\newblock {\em SIAM/ASA J. Uncertain. Quantif.}, 7(3):948--974, 2019.

\bibitem{riis2023}
Nicolai Riis, Amal Alghamdi, Felipe Uribe, Silja Christensen, Babak Afkham,
  Per~Christian Hansen, and Jakob Jorgensen.
\newblock Cuqipy -- part i: computational uncertainty quantification for
  inverse problems in python.
\newblock {\em arXiv preprint arXiv:2305.16949}, 05 2023.

\bibitem{cui2016}
Tiangang Cui, Kody J.~H. Law, and Youssef~M. Marzouk.
\newblock Dimension-independent likelihood-informed {MCMC}.
\newblock {\em J. Comput. Phys.}, 304:109--137, 2016.

\bibitem{santosa1995}
Fadil Santosa.
\newblock A level-set approach for inverse problems involving obstacles.
\newblock {\em ESAIM Contr\^{o}le Optim. Calc. Var.}, 1:17--33, 1995/96.

\bibitem{kocyigit2017}
Ilker Kocyigit, Ru-Yu Lai, Lingyun Qiu, Yang Yang, and Ting Zhou.
\newblock Applications of {CGO} solutions to coupled-physics inverse problems.
\newblock {\em Inverse Probl. Imaging}, 11(2):277--304, 2017.

\bibitem{clop2010}
Albert Clop, Daniel Faraco, and Alberto Ruiz.
\newblock Stability of {C}alder\'{o}n's inverse conductivity problem in the
  plane for discontinuous conductivities.
\newblock {\em Inverse Probl. Imaging}, 4(1):49--91, 2010.

\bibitem{alessandrini2005}
G.~Alessandrini and M.~Di~Cristo.
\newblock Stable determination of an inclusion by boundary measurements.
\newblock {\em SIAM J. Math. Anal.}, 37(1):200--217, 2005.

\bibitem{devore1993}
Ronald~A. DeVore and Robert~C. Sharpley.
\newblock Besov spaces on domains in {${\bf R}^d$}.
\newblock {\em Trans. Amer. Math. Soc.}, 335(2):843--864, 1993.

\bibitem{engelking1989}
Ryszard Engelking.
\newblock {\em General topology}, volume~6 of {\em Sigma Series in Pure
  Mathematics}.
\newblock Heldermann Verlag, Berlin, second edition, 1989.
\newblock Translated from the Polish by the author.

\end{thebibliography}

\medskip
\medskip

\end{document}